\documentclass[12pt]{amsart}
\usepackage{graphicx, amsfonts, amssymb}
\usepackage{mathrsfs, amsmath, amsthm}
\usepackage{verbatim}

       \def\b{\beta}        \def\g{\gamma}  \def\ep{\epsilon}

\def\a{\alpha}
\def\ep{\varepsilon}

\newcommand{\hol}{{\mathcal Hol}}
\DeclareMathOperator{\og}{O}

\def\D{{\mathbb D}}
\def\T{{\mathbb T}}
\def\C{{\mathbb C}}  \def\N{{\mathbb N}}
 \def\R{{\mathbb R}}

\def\Dp{{\mathcal D^p_{p-1}}}

\def\Dpa{{\mathcal D^p_{\alpha}}}
\def\Dqb{{\mathcal D^q_{\beta}}}

\def\Dq{{\mathcal D^q_{q-1}}}

\def\({\left(}       \def\){\right)}

\DeclareMathOperator{\supp}{supp}
\newcommand{\iii}{\frac 1{2\pi}\int_0^{2\pi}}
\newcommand{\ig}{\stackrel{\text{def}}{=}}

\newcommand{\hti}{\widetilde{\mathcal{H}}}
\newcommand{\hg}{\mathcal{H}_{g}}
\newcommand{\hgp}{\mathcal{H}_{g'}}
\newcommand{\n}[1]{\|#1\|}
\newtheorem{theorem}{Theorem}
\newtheorem{lemma}{Lemma}
\newtheorem{corollary}{Corollary}

\theoremstyle{definition}

\theoremstyle{remark}
\newtheorem{remark}{Remark}
\numberwithin{equation}{section}
\theoremstyle{theorem}
\newtheorem{other}{\bf Theorem}              
\newtheorem{otherl}{\bf Lemma}        

\newenvironment{pf}{\noindent{\emph{Proof.}}}{$\Box$ }
\newenvironment{Pf}{\noindent{\emph{Proof of}}}{$\Box$ }
 \DeclareMathOperator{\ogr}{O}

\begin{document}
\title[Generalized Hilbert operators]
{Generalized Hilbert operators }

\author[P.~Galanopoulos]{Petros Galanopoulos}
 \address{Department of Mathematics,
Aristotle University of Thessaloniki, 54124, Thessaloniki, Greece}
\author[D.~Girela]{Daniel Girela}
 \address{Departamento de An\'alisis Matem\'atico,
Universidad de M\'alaga, Campus de Teatinos, 29071 M\'alaga, Spain}
 \email{girela@uma.es}
\author[J.~A.~Pel\'aez]{Jos\'e \'Angel Pel\'aez}
 \address{Departamento de An\'alisis Matem\'atico,
Universidad de M\'alaga, Campus de Teatinos, 29071 M\'alaga, Spain}
 \email{japelaez@uma.es}
 \author[A.~Siskakis]{Aristomenis G. Siskakis}
 \address{Department of Mathematics,
Aristotle University of Thessaloniki,
54124, Thessaloniki,
Greece}
 \email{siskakis@math.auth.gr}
\date{August 30, 2012}
\begin{abstract}
If $g$ is an analytic function in the unit disc $\D $ we consider
the generalized Hilbert operator $\hg$
 defined by
\begin{equation*}\label{H-g}
\mathcal{H}_g(f)(z)=\int_0^1f(t)g'(tz)\,dt.
\end{equation*}
We study these operators acting on classical spaces of analytic
functions in $\D $. More precisely, we address the question of
characterizing the functions $g$ for which the operator $\hg $ is
bounded (compact) on the Hardy spaces  $H^p$, on the weighted
Bergman spaces $A^p_\alpha $ or on the spaces of Dirichlet type
$\mathcal D^p_\alpha $.
\end{abstract}
\thanks{The second and third author are supported by grants from \lq\lq El Ministerio de Econom\'{\i}a y Competitividad,
Spain\rq\rq\, MTM2011-25502;  and from \lq\lq La Junta de
Andaluc{\'\i}a\rq\rq\, (FQM210) and (P09-FQM-4468). The third author
is also supported by the Ram\'on y Cajal program, MICINN (Spain).}
\maketitle


\section{Introduction}\label{intro}
\subsection{Generalized Hilbert operators}
We denote by $\D $ the unit disc in the complex plane $\mathbb{C}$, and by $\hol (\D )$ the space of all analytic functions in $\D $.
\par
The Hilbert matrix
$$H=\left(%
\begin{array}{ccccc}
            1 & \frac{1}{2}  & \frac{1}{3}  & .  \\
  \frac{1}{2} & \frac{1}{3}  & \frac{1}{4}  & .  \\
  \frac{1}{3} & \frac{1}{4}  & \frac{1}{5}  & . \\
  .  & . & . & .  \\
\end{array}%
\right),
$$
can be viewed as an operator on spaces of analytic functions, called
the \,{\it Hilbert operator},\, by its action on the Taylor coefficients:
$$
a_n\mapsto \sum_{k=0}^{\infty}
\frac{a_k}{n+k+1}, \quad n=0,1,2, \cdots,
$$
that is, if
 $f(z)=\sum_{k=0}^\infty a_kz^k\in \hol (\D )$
we define
\begin{equation}\label{H}
\mathcal{H}(f)(z)=
\sum_{n=0}^{\infty}\left(\sum_{k=0}^{\infty}
\frac{a_k}{n+k+1}\right)z^n,
\end{equation}
whenever the  right hand side makes sense and defines and
analytic function in $\D $.
\par
Hardy's inequality \cite[page~48]{D} guarantees that the transformed
power series in (\ref{H}) converges on $\mathbb{D}$ and defines
there an analytic function $\mathcal{H}(f)(z)$ whenever $f\in H^1$.
In other words, $\mathcal{H}(f)$ is a well defined analytic function
for every $f\in H^1$.
\par
It turns out that
$\mathcal{H}(f)$ can be written also in the form,
\begin{equation*}\label{H-int}
\mathcal{H}(f)(z)=\sum_{n=0}^{\infty}\left(\int_0^1
t^n f(t)\,dt\right)z^n=\int_0^1f(t)\frac{1}{1-tz}\,dt,
\end{equation*}
or, equivalently,
$$
\mathcal{H}(f)(z)=\int_0^1f(t)g'(tz)\,d\zeta,
$$
where $g(z)=\log\frac{1}{1-z}$.
\par
The resulting Hilbert operator $\mathcal{H}$ is bounded from $H^p$
to $H^p$, whenever $1<p<\infty$ but
  $\mathcal{H}$ is not bounded on $H^1$
\cite[Theorem 1.1]{DiS}. In \cite{DJV}  the norm of $\mathcal{H}$ acting on Hardy spaces was  computed. Concerning the Bergman
spaces $A^p$, the operator $\mathcal{H}: A^p\to A^p$ is bounded if
and only if $2<p<\infty$, \cite{Di}. But $\mathcal{H}$ is not
 even defined in $A^2$, for  it was shown  in  \cite{DJV} that
there exist functions $f\in A^2$ such that the series defining
$\mathcal{H}(f)(0)$ is divergent.
\par\medskip
 In this article we shall be dealing with certain generalized Hilbert operators.
 Given $g\in\hol (\D )$, we consider the generalized Hilbert operator $\hg$ defined by
\begin{equation}\label{H-g}
\mathcal{H}_g(f)(z)=\int_0^1f(t)g'(tz)\,dt.
\end{equation}
As noted above, $\mathcal H=\mathcal{H}_g$ with
$g(z)=\log\frac{1}{1-z}$. We mention \cite{GaPe2010} for a different
generalization of the classical Hilbert operator.
\par
The Fej\'{e}r-Riesz inequality \cite[page 46]{D} guarantees that
given any $g\in\hol (\mathbb D)$, the integral in (\ref{H-g})
converges absolutely, and therefore the right hand side of
(\ref{H-g}) defines an analytic function on $\mathbb{D}$, for every
$f\in H^1$.
\par
We note that $\hg$ has a representation in terms of the Taylor coefficients
similar to (\ref{H}). Indeed, a simple computation shows that if
$g(z)=\sum_{n=0}^\infty b_nz^n\in \hol (\D )$ and
$f(z)=\sum_{n=0}^\infty a_nz^n\in H^1$ then
\begin{equation}
\begin{split}\label{Hgcoef}
\mathcal{H}_g(f)(z)&=\sum_{k=0}^{\infty} \left( (k+1)b_{k+1}\int_0^1t^k f(t)\,dt\right)z^k\\
&=\sum_{k=0}^\infty \left
((k+1)b_{k+1}\sum_{n=0}^\infty \frac{a_n}{n+k+1}\right
)z^k.
\end{split}
\end{equation}
\par
Our main objective in this paper is characterizing those  functions
$g$ for which  $\mathcal H_g$ is bounded on the Hardy spaces $H^p$,
the Bergman spaces $A^p_\alpha$ and on the the spaces of Dirichlet
type $\Dpa $ ($0<p<\infty , \alpha >-1$). These results are stated
in Section \ref{Main}.
\subsection{Spaces of analytic functions}
If $\,0<r<1\,$ and $\,f\in \hol (\D)$, we set
$$
M_p(r,f)=\left(\iii |f(re^{it})|^p\,dt\right)^{1/p}, \,\,\,
0<p<\infty ,
$$
$$
M_\infty(r,f)=\sup_{\vert z\vert =r}|f(z)|.
$$
\subsubsection{Hardy and Bergman spaces.} If $\,0<p\le \infty $,\, the Hardy space $H^p$ consists of those
$f\in \hol(\mathbb D)$ such that $\n{ f} _{H^p}\ig
\sup_{0<r<1}M_p(r,f)<\infty $. Functions $f$ in Hardy spaces have non-tangential boundary
values $f(e^{i\theta})$ almost everywhere on the unit circle $\T$.
\par
If $0<p<\infty $ and $\a>-1$, the weighted
Bergman space $A^p_\a$ consists of those $f\in \hol(\mathbb D)$
such that
$$
\n{ f} _{A^p_\alpha }\ig \left ((\alpha +1) \int_\D
| f(z)|^p(1-|z|^2 )^{\a}\, dA(z) \right
)^{1/p}<\infty .
$$
The unweighted Bergman space $A\sp p\sb 0 $ is simply denoted by
$A\sp p $. Here, $dA(z) =\frac{1}{\pi}dx\,dy $ denotes the
normalized Lebesgue area measure in $\mathbb D$. For each $p\in
(0,\infty )$ the Hardy space $H^p$ is contained in $A^{2p}$ and the
exponent $2p$ cannot be improved.  We refer to   \cite{D} for the
theory of Hardy spaces, and to \cite{DS}, \cite{HKZ} and \cite{Zhu}
for Bergman spaces.
\subsubsection{Dirichlet type spaces. } If $0<p<\infty $ and $\a>-1$ the
space of Dirichlet type $\Dpa$
consists of all indefinite integrals of functions in $A^p_\alpha$.  Hence, if $f$ is analytic in $\mathbb D$, then $f\in \Dpa$ if and
only if
$$
\n{f}_{\Dpa}^p\ig |f(0)|^p+\n{f'}_{A^p_\alpha }^p <\infty .
$$
The space $\mathcal D^2_0$ is the classical Dirichlet space
$\mathcal D$ and  $D^2_1=H^2$. For each $p$, the range of values of
the parameter $\a$ for which $\Dpa$ is most interesting is
$$
p-2\leq \a \leq p-1.
$$
 If $\a >p-1$ then it is easy to see that
$\Dpa =A^p_{\alpha -p}$. Indeed this follows from the well known estimate
$$
\int_{\D}|f(z)|^p(1-|z|)^s\,dA(z)\asymp
|f(0)|^p+\int_{\D}|f'(z)|^p(1-|z|)^{p+s}\,dA(z),
$$
(see, e.\,\@g., \cite[Theorem\,\@6]{Flett}). On the other hand, if
$\a <p-2$ then $\frac{\a+2-p}{p}<0$ and then it follows easily that
$\Dpa \subset H^\infty $ in this case. For $\a=p-2$ the space
$\mathcal D^p_{p-2}$ coincides with  the Besov space usually denoted
by $B^p$.
\par
For $\a=p-1$ the space $\Dp$ is the closest to the Hardy space $H^p$
but does not coincide with it for $p\ne2$. If $0<p\le 2$ then $\Dp
\subset H\sp p$ \cite{Flett} and if $2\le p<\infty $ then $H\sp
p\subset \Dp$ \cite{LP}.
 \subsubsection{Mean Lipschitz spaces.} We shall consider also the mean Lipschitz spaces
$\Lambda\left(p,\alpha\right)$. For $1\le p<\infty$ and $0<\a\le 1$ the
space $\Lambda\left(p,\a\right)$  consists of those $g\in \hol(\D)$ having
a non-tangential limit $g(e^{i\theta})$ almost everywhere and such
that
$$
\omega_p(g, t)=\ogr(t^{\alpha}), \quad t\to 0,
$$
where
$$
\omega_p(g, t)=\sup_{0<h\le t}\left(\int_0^{2\pi}
|g(e^{i(\theta+h)})-g(e^{i\theta})|^p \frac{d\theta}{2\pi}\right)^{1/p}
$$
is the integral modulus of continuity of order $p$.
A classical result of Hardy and Littlewood
\cite{HL-32} (see also Chapter~5 of \cite{D}) asserts that
\begin{equation}\label{hl32}
\Lambda\left(p,\a\right)=\left \{ f\in H^p : M_p(r,f^\prime)
=\og \left ((1-r)^{\alpha -1}\right )\right \},
\end{equation}
for $1\leq p<\infty
,\,\,0<\a \leq 1$. The corresponding \lq\lq little
oh\rq\rq \, spaces  are denoted by $\lambda (p, \alpha )$.
\par Among all the mean Lipschitz spaces, the spaces $\Lambda (p,\frac{1}{p})$, $1<p<\infty $, will play a fundamental role
in our work. They form a nested scale of spaces which are all
contained in the space $BMOA$ \cite{BSS}:
$$
\Lambda \left(q,\frac{1}{q}\right)\subset \Lambda  \left(p,\frac{1}{p}\right)\subset BMOA,\quad 1\le q<p<\infty .
$$
Furthermore the function $\log(\frac{1}{1-z})$ belongs to
$\Lambda\left (p,\frac{1}{p}\right)$ for each $p>1$.
\section{Main results}\label{Main}
 Our main results regarding Hardy spaces are contained in
Theorem~\ref{th:hardy1p2} and Theorem~\ref{th:hardy2pinfty}.
\begin{theorem}\label{th:hardy1p2} Suppose that $1<p\le 2$ and
$g\in \hol(\D)$. Then $\hg$ is bounded from $H^p$ to
$H^p$ if and only if $g\in \Lambda\left(p,\frac{1}{p}\right)$.
\end{theorem}
\begin{theorem}\label{th:hardy2pinfty} Suppose that $2<p<\infty $ and
$g\in \hol(\D)$. We have:
\par (i)\, If $\hg$ is bounded from $H^p$ to $H^p$,
then $g\in \Lambda\left(p,\frac{1}{p}\right)$.
\par (ii)\, If  $g\in \Lambda\left(q,\frac{1}{q}\right)$ for some $q$ with $1<q<p$, then
$\hg$ is bounded from $H^p$ to $H^p$.
\end{theorem}
\par\medskip
It is natural to ask whether or not the condition
 $g\in \Lambda\left(p,\frac{1}{p}\right)$
implies that $\hg$ is bounded from $H^p$ to $H^p$, for $2<p<\infty
$. We do not know the answer to this question but we conjecture that
it is affirmative. The condition $g\in
\Lambda\left(q,\frac{1}{q}\right)$ for some $q$ with $1<q<p$ which
appears in (ii) is slightly stronger than that of $g$ belonging to
$\Lambda \left (p,\frac{1}{p}\right )$.
\par
Using (\ref{hl32}) it follows that if $g\in \hol (\D )$
has power series  $g(z)=\sum_{k=0}^\infty b_kz^k$ with
$\sup_{k\in\N}k|b_{k}|<\infty $, then $g\in \Lambda \left
(2,\frac{1}{2}\right )$. Also, using (\ref{hl32}) and the Littlewood
subordination principle, it follows easily that a function $g\in
\hol (\D )$ such that $\Re (g'(z))\geq 0$, for all $z\in \D $,
belongs to $\Lambda \left (q,\frac{1}{q}\right )$ for all $q>1$, a
result which readily implies that the same is true for any $g\in
\hol (\D )$ which is the Cauchy transform of a finite, complex,
Borel measure $\mu $ on the circle $\mathbb T$, that is,
$$
g(z)=\int_{\mathbb T}\frac{d\mu (\zeta )}{1-\bar\zeta z}.
$$
Consequently, it is clear that we have  the following.
\begin{corollary}\label{co:substitute}
Let $\mathcal{K}$ be the class of those analytic functions in $\D $
which are the Cauchy transform of a finite, complex, Borel measure
on $\mathbb T$ and let
$$
\mathcal{C}=\left\{g(z)=\sum_{k=0}^\infty
b_kz^k\in\hol(\D):\,\sup_{k\in\N}k|b_{k}|<\infty\right\}.
$$ We have:
\par (i)\, If $2\leq p<\infty$ and  $g\in\mathcal{C}$, then $\hg:\,H^p\to H^p$ is
bounded.
\par (ii)\, If $1<p<\infty $ and $g\in \mathcal K$, then
$\hg:\,H^p\to H^p$ is bounded.
\end{corollary}
\par
We note that $\mathcal{K}$ and $\mathcal{C}$ are  subclasses of the
mentioned mean Lipschitz spaces containing the function
$g(z)=\log\frac{1}{1-z}$. Thus, Corollary~\ref{co:substitute}
generalizes the classical result on the boundedness of the Hilbert
operator on $H^p$.
\par\medskip
It turns out that $g\in \Lambda \left (p,\frac{1}{p}\right )$ is equivalent to the
boundedness of the operator $\hg $ on the weighted Bergman spaces
$A^p_\alpha $ and on the spaces of Dirichlet type $\mathcal
D^p_\alpha $ for the admissible values of $p$ and $\alpha $.
\begin{theorem}\label{th:bergman}
Suppose that $1<p<\infty$, $-1<\alpha<p-2$ and  $g\in \hol(\D)$. Then
 $\hg:A^p_\alpha\to A^p_\alpha$ is bounded  if
 and only if $g\in \Lambda\left(p,\frac{1}{p}\right)$.
\end{theorem}
\par
The condition $-1<\alpha<p-2$ is not a real restriction. It is needed to insure that any function
$f\in A^p_\alpha$ satisfies that $\int_0^1|f(t)|\,dt<\infty$, which
is necessary for the operator $\hg$ being well defined on
$A^p_\alpha $. The result does not remain true for $\alpha\ge p-2$.
\medskip
\begin{theorem}\label{th:dirichlet}
Suppose that $1<p<\infty$, $p-2<\alpha\le p-1$ and  $g\in \hol(\D)$. Then
 $\hg:\Dpa\to \Dpa$ is bounded  if
 and only if $g\in \Lambda\left(p,\frac{1}{p}\right)$.
\end{theorem}
\par\medskip
The paper is organized as follows. In Section~\ref{preliminary} we
state and prove a number of lemmas which will be used specially in
Section~\ref{necessity} where we shall prove the necessity parts of
our just mentioned results. Section~\ref{sublinear} will be devoted
to study the {\em{the sublinear Hilbert operator}} $\hti $ defined
by
 $$
 \hti(f)(z)=\int_0^1\frac{|f(t)|}{1-tz}\,dt.
 $$
We shall prove that if $g\in \Lambda\left(p,\frac{1}{p}\right)$ and
$X$ is either $H^p$\,with\, $1<p\le 2$, or
$A^p_\alpha$\,with\,$1<p<\infty $ and $-1<\alpha<p-2$, or $\Dpa$
\,with\,$1<p<\infty $ and $p-2<\alpha\le p-1$, then
$$
||\hg(f)||_{X}\le C||\hti(f)||_{X},\quad f\in X.
$$
The sufficiency parts of our Theorems~1-4 will follow using this and
the following result which has independent interest.
 \begin{theorem}\label{th:hti}
 \par (i)\, If $p>1$, then\, $\hti: H^p\to H^p$ is bounded.
 \par (ii)\, If $p>1$ and $-1<\alpha<p-2$, then \, $\hti: A^p_\alpha\to A^p_\alpha$ is bounded.
 \par (iii)\, If $p>1$ and $p-2<\alpha\le p-1$, then \, $\hti: \Dpa\to \Dpa$ is bounded.
 \end{theorem}
 \par\medskip
 In Section~\ref{compactness} we shall deal with the question
 of characterizing the functions $g$ for which $\hg$ is compact on
 Hardy, Bergman and Dirichlet spaces. We prove the \lq\lq expected
 results\rq\rq , that is, Theorems~1-4 remain true if we change \lq\lq
 bounded\rq\rq \, to \lq\lq compact\rq\rq \, and the mean Lipschitz
 space $\Lambda (s,\alpha )$ appearing there to the corresponding \lq\lq
 little oh\rq\rq\, space $\lambda (s,\alpha )$.
We also obtain the characterization of the functions $g$ for which
the operator $\mathcal{H}_g$ is Hilbert-Schmidt on the relevant Hilbert spaces.
\begin{theorem}\label{th:sachatten}
The following are equivalent\\
(i) $\mathcal{H}_g$ is Hilbert-Schmidt on $H^2$.\\
(ii) $\mathcal{H}_g$ is Hilbert-Schmidt on $A^2_\a$ for any  $-1<\a<0$.\\
(iii) $\mathcal{H}_g$ is Hilbert-Schmidt on $\mathcal{D}^2_\a$ for any $0<\a\leq 1$. \\
(iv) $g\in \mathcal{D}$.
\end{theorem}
 Note that the case $\a=1$ of (iii) is assertion (i).
\par\medskip
 We close this section
noticing that, as usual, we shall be using the convention that
$C=C(p, \alpha ,q,\beta ) \dots$ will denote a positive constant
which depends only upon the displayed parameters $p, \alpha , q,
\beta \dots $ (which sometimes will be omitted) but not  necessarily
the same at different occurrences.
Moreover, for two real-valued functions $E_1,E_2$ we write $E_1\asymp
E_2$, or $E_1\lesssim E_2$, if there exists a positive constant $C$
independent of the argument such that $\frac{1}{C} E_1\leq E_2\leq C
E_1$, respectively $E_1\leq C E_2$.
\par\medskip
\section{Preliminary results}\label{preliminary}
Throughout the paper we shall use the following notation:
If $g(z)=\sum_{k=0}^\infty b_k z^k\in\hol(\D)$ and $n\ge 0$, we set
$$
\Delta_ng(z)=\sum_{k\in I(n)}b_kz^k
$$
where $I(n)=\left\{k\in\N:\, 2^n\le k\le 2^{n+1}-1 \right\}.$
\par Let us recall several distinct characterizations of  $\Lambda(p,\alpha)$
spaces, (see \cite{BSS}, \cite{D}, \cite{GPP} and \cite{MP}).
\begin{other}\label{th:mlip}
Suppose that $1<p<\infty$, $0<\alpha<1$ and $g\in \hol(\D)$. The
following conditions are equivalent
\par (i)\, $g\in \Lambda(p,\a).$
\par (ii)\, $M_p(r,g')=\ogr\left(\frac{1}{(1-r)^{1-\alpha}}\right)$,\,\,\,\,\,as $r\to 1^-$.
\par (iii) \,$||\Delta_ng||_{H^p}=\ogr\left(2^{-n\alpha}\right)$,\quad\,\,\,\, as $n\to\infty$.
\par (iv)\, $||\Delta_ng'||_{H^p}=\ogr\left(2^{n(1-\alpha)}\right)$,\quad as $n\to\infty$.
\par (v)\, $||\Delta_ng''||_{H^p}=\ogr\left(2^{n(2-\alpha)}\right)$,\quad as $n\to\infty$.
\end{other}
\begin{remark}\label{re:mlip}
The corresponding results  for the little-oh space
$\lambda(p,\alpha)$ remain true, and they  can be proved following the
proofs in the  references for Theorem \ref{th:mlip}.
\end{remark}
\medskip
Suppose $W(z)=\sum_{k\in J}b_kz^k$ is a polynomial,
so $J$ is a finite subset of $\mathbb{N}$, and $ f(z)=\sum_{k=0}^{\infty}a_kz^k\in \hol(\D)$.
We consider the Hadamard product
$$
(W\ast f)(z)=\sum_{k\in J}b_ka_kz^k,
$$
and observe that if $f\in H^1$ then
$$
(W\ast f)(e^{it})=\frac{1}{2\pi}\int_0^{2\pi}W(e^{i(t-\theta)})f(e^{i\theta})\,d\theta
$$
is the usual convolution.
\par\medskip
If $\Phi:\R\to\C$ is a $C^\infty$-function such that $\supp(\Phi)$  is a compact subset of  $(0, \infty)$
we set
$$
A_\Phi=\max_{s\in\R}|\Phi(s)|+\max_{s\in\R}|\Phi''(s)|,$$
and for $N=1,2,\dots$, we consider  the polynomials
$$
W_N^\Phi(z)=\sum_{k\in\mathbb N}\Phi\left (\frac{k}{N}\right )z^k.
$$
\par
Now, we are ready to state the next result on smooth  partial sums.
\begin{other}\label{th:cesaro}
Assume that $\Phi:\R\to\C$ is a $C^\infty$-function with $\supp(\Phi)$ a compact set
contained in $(0, \infty)$. Then
\begin{itemize}\item[(i)]
 There exists an absolute constant $C>0$ such that if\, $m\in \{ 0,
1, 2, \dots \} $ and $N\in \{ 1, 2, 3, \dots \} $ then
$$
\left|W_N^\Phi(e^{i\theta})\right|\le C\min\left\{
N\max_{s\in\R}|\Phi(s)|,
N^{1-m}|\theta|^{-m}\max_{s\in\R}|\Phi^{(m)}(s)|
\right\},
$$
for $0<|\theta|<\pi.$
\item[(ii)] There exists a positive  constant $C$ such
that
$$
\left|(W_N^\Phi\ast f)(e^{i\theta})\right|\le CA_\Phi
M(|f|)(e^{i\theta}),\quad \text{for all $f\in H^1$},
$$
where $M$ is the Hardy-Littlewood maximal-operator, that is,
$$
M(|f|)(e^{i\theta})=\sup_{0<h<\pi}\frac{1}{2h}\int_{\theta-h}^{\theta+h}|f(e^{it})|\,dt.
$$
\item[(iii)]
For every $p\in (1,\infty)$ there exists $C_p>0$ such that
$$
||W_N^\Phi\ast f||_{H^p}\le C_p A_\Phi ||f||_{H^p},\quad f\in H^p.
$$
\item[(iv)]
For every $p\in (1,\infty)$ and $\a>-1$ there is $C_p>0$ such that
$$
||W_N^\Phi\ast f||_{A^p_\a}\le C_p A_\Phi ||f||_{A^p_\a},\quad f\in A^p_\alpha.
$$
\end{itemize}
\end{other}
\par
 Theorem~\ref{th:cesaro} follows from
the results and proofs in \cite[p. $111-113$]{Pabook}.
\par
\medskip The following lemma also plays an  essential role in our work.
\begin{lemma}\label{le:1}
Suppose that\, $1<p<\infty $ and $\alpha>-1$. For $N=1, 2, \dots$,
let $a_N=1-\frac1N$ and define the functions
\begin{equation}\label{eq:psialpha}
\psi_{N,\a}(s)=\frac{1}{N^{3-\frac{2+\a}{p}}}\int_0^1\frac{t^{sN}}{(1-a_Nt)^2}\,dt,\quad
s>0.
\end{equation}
and
\begin{equation}\label{eq:psialpha2}
\varphi_{N,\a}(s)=\frac{1}{\psi_{N,\a}(s)}\quad
s>0.
\end{equation}
Then: \begin{itemize} \item[(i)]$\psi_{N,\a}, \varphi_{N,\a} \in
C^\infty((0,\infty))$. \item[(ii)] Asymptotically,
$$
|\psi_{N,\a}(s)|\asymp\frac{1}{N^{2-\frac{2+\a}{p}}},\quad
\frac12<s<4,\quad N\to \infty.
$$
\item[(iii)]For each $m\in\N$ there is a constant $C(m)>0$ (depending on $m$ but
not on $N$) such that
$$
|\psi^{(m)}_{N,\a}(s)|\le\frac{C(m)}{N^{2-\frac{2+\a}{p}}},\quad
\frac12<s<4,\quad N=1,2, \dots.
$$
\item[(iv)]For each $m\in\N$ there is a constant $C(m)>0$ (depending on $m$ but
not on $N$) such that
\begin{equation}\label{eq:psialpha3}
|\varphi^{(m)}_{N,\a}(s)|\le C(m)N^{2-\frac{2+\a}{p}},\quad
\frac12<s<4,\quad N=1,2,  \dots.
\end{equation}
\end{itemize}
\end{lemma}
\begin{proof}
\par (i)\, is clear.
\par (ii).\, We note that
 \begin{equation*}
\int_0^1\frac{t^{sN}}{(1-a_Nt)^2}\,dt \le
\int_{0}^1\frac{1}{(1-a_Nt)^2}\,dt=N,
\end{equation*}
while if $\frac12<s<4$, then
\begin{align*}
\int_0^1\frac{t^{sN}}{(1-a_Nt)^2}\,dt &\ge
\int_{a_N}^1\frac{t^{sN}}{(1-a_Nt)^2}\,dt\\
 &\ge  (1-a_N)\frac{a_N^{4N}}{(1-a_N^2)^2}\\
 &=\frac{(1-\frac{1}{N})^{4N}}{(2-\frac{1}{N})^2} N\\
 &\ge CN,
\end{align*}
so for $\frac12<s<4$,
$$
|\psi_{N,\a}(s)|\asymp\frac{1}{N^{2-\frac{2+\a}{p}}},\quad
 \mbox{as}\,\,N\to\infty.
$$
\par (iii).\, Since
$$
\sup_{0<t<1,\,\frac12<s<4}\left(\log
\frac{1}{t^N}\right)^m t^{sN}\le \sup_{0<x<1}\left(\log
\frac{1}{x}\right)^m x^{1/2}=C(m)<\infty,
$$
we deduce that
\begin{equation*}\begin{split}
|\psi^{(m)}_{N,\a}(s)& |=
 \frac{1}{N^{3-\frac{2+\a}{p}}}\int_0^1\frac{\left(\log \frac{1}{t^N}\right)^m
t^{sN}}{(1-a_Nt)^2}\,dt
\\ & \le C(m) \frac{1}{N^{3-\frac{2+\a}{p}}}\int_0^1\frac{1}{(1-a_Nt)^2}\,dt
\\ & \le C(m) \frac{1}{N^{2-\frac{2+\a}{p}}},\quad
\frac12<s<4,\quad N=1,2, \dots.
\end{split}\end{equation*}
\par (iv).\,
For $m=0$, the assertion follows from part (ii). For $m=1$, using parts (ii)  and (iii), we have
$$
|\varphi'_{N,\a}(s)|\le \frac{|\psi'_{N,\a}(s)|}{|\psi_{N,\a}(s)|^2}\le C(1) N^{2-\frac{2+\a}{p}},\quad
\frac12<s<4.
$$
\par Now we shall proceed by induction. Assume that  (\ref{eq:psialpha3}) holds for $j=0,1\dots m-1$. Since
$1=\varphi_{N,\a}(s)\psi_{N,\a}(s)$, we have
$$
0=(\varphi_{N,\a}(s)\psi_{N,\a}(s))^{(m)}(s)=\sum_{j=0}^m
\binom{m} {j} \psi^{(m-j)}_{N,\a}(s)\varphi^{(j)}_{N,\a}(s),
$$
 which implies
$$
|\varphi_{N,\a}^{(m)}(s)|\le\frac{\sum_{j=0}^{m-1}\binom{m} {j}\left|\psi^{(m-j)}_{N,\a}(s)\varphi^{(j)}_{N,\a}(s)\right|}{|\psi_{N,\a}(s)|},\quad
\frac12<s<4.
$$
 This  together with the induction hypothesis and part (iii) concludes the proof.
\end{proof}
\par\medskip
We shall use also the following lemma which follows easily from
results in \cite{MP}.
\begin{lemma}\label{le:eqno}
Assume that $0<p<\infty$, $\alpha>-1$, $N\in\mathbb{N}$, and set
$$
h(z)=\sum_{N/2\leq k \leq 4N}a_kz^k.
$$
Then
$$
||h||_{A^p_\a}\asymp N^{-\frac{1+\alpha}{p}}||h||_{H^p}.
$$
\end{lemma}
\begin{proof}
Assume $N$ is even. (If $N$ is odd the proof can be adjusted by using $[\frac{N}{2}]+1$ instead of $\frac{N}{2}$).
Using  \cite[Lemma 3.1]{MP} we have for each $0<r<1$,
$$
||h||_{H^p}^p r^{p4N}\leq M_p^p(r, h)\leq ||h||_{H^p}^p r^{p\frac{N}{2}}
$$
which gives
$$
||h||_{H^p}^p \int_0^1r^{p4N+1}(1-r^2)^{\a}dr\leq \frac{||h||_{A^p_{\a}}^p}{\a+1}\leq ||h||_{H^p}^p\int_0^1 r^{p\frac{N}{2}+1}(1-r^2)^{\a}dr.
$$
 Each of the two integrals appearing above can be expressed in terms of the usual  Beta function, and using the
 Stirling asymptotic series we can see  that each of the integrals grows as $N^{-(\a+1)}$ as $N\to \infty$, and the assertion follows.
\end{proof}
\par\medskip
\section{Necessary conditions for the boundedness of $\hg$}\label{necessity}
Putting together the conditions stated in Theorem~\ref{th:hardy1p2}
and Theorem~\ref{th:hardy2pinfty} as necessary for the boundedness
of the operator on $H^p$ for $1<p\le 2$ and $2<p<\infty$,
respectively, yields the following statement.
\begin{theorem}\label{th:hardy.nec} Suppose that $1<p<\infty $ and
$g\in \hol(\D)$. If $\hg$ is bounded from $H^p$ to
$H^p$, then $g\in \Lambda\left(p,\frac{1}{p}\right )$.
\end{theorem}
\medskip\par
\begin{pf}
Let  $g(z)=\sum_{k=0}^\infty b_kz^k$ be the Taylor expansion of $g$. We start by considering the function
$\psi_{N, \alpha}$  and $\phi_{N, \a}=\frac{1}{\psi_{N, \a}}$ defined in Lemma \ref{le:1}
with  $\alpha=p-1$  and, for simplicity, write $\psi_N=\psi_{N, p-1} $ and $\varphi_{N}=\varphi_{N,p-1}$.
\par
For each $N=1,2,\dots$, we can find  a $C^\infty$-function
 $\Phi_N:\R\to\C$ with $\supp\left(\Phi_N\right)\subset \left(\frac12, 4 \right)$, satisfying
\begin{equation}\label{eq:nh1}
\Phi_N(s)=\varphi_N(s),\quad 1\le s\le 2,
\end{equation}
and such that, by using part (iv) of Lemma \ref{le:1},  for each $m\in\N$ there exists $C(m)$ (independent of $N$) with
\begin{equation}\label{eq:nh2}
|\Phi^{(m)}_N(s)|\le C(m)N^{1-\frac{1}{p}},\quad
s\in\R,\quad N=1,2,\dots.
\end{equation}
In particular we have
\begin{equation}\label{eq:nh2-1}
A_{\Phi_N}= \max_{s\in\R}|\Phi_N(s)|+\max_{s\in\R}|\Phi_N''(s)|\leq CN^{1-\frac{1}{p}}
\end{equation}
\par
Let us consider now the family of test functions $\{ f_N \}$ given by
$$
f_N(z)=\frac{1}{N^{2-\frac1p}}\frac{1}{(1-a_Nz)^2},\quad z\in\D,\,\, N=1, 2, \dots
$$
An easy calculation using  \cite[Lemma, page 65]{D}) shows that the
$H^p$-norms of the functions $f_N$ are  uniformly bounded. By the
hypothesis
$$
\sup_{N}||\hg(f_N)||_{H^p}=C<\infty.
$$
This, together with part (iii) of Theorem~\ref{th:cesaro} and
(\ref{eq:nh2-1}), implies
\begin{equation}\begin{split}\label{eq:nh3}
||W_N^{\Phi_N}\ast \hg(f_N)||_{H^p} & \le C_p A_{\Phi_N} ||\hg(f_N)||_{H^p}\le C_p N^{1-\frac1p}.
\end{split}\end{equation}
\par On the other hand,
\begin{align*}
(W_N^{\Phi_N}\ast \hg(f_N))(z)&=\sum_{\frac{N}{2}\le k\le 4N}\left[(k+1)b_{k+1}\left(\int_0^1t^kf_N(t)\,dt\right)\Phi_N(\frac{k}{N})\right]z^k\\
&=\sum_{\frac{N}{2}\le k\le N-1}[\cdots]z^k+\sum_{N\le k\le 2N-1}[\cdots]z^k +\sum_{2N\le k\le 4N }[\cdots]z^k\\
&=F_1^N(z)+F_2^N(z)+F_3^N(z)
\end{align*}
and by (\ref{eq:nh1})
\begin{equation}\begin{split}\label{eq:nh4}
F_2^N(z)&=\sum_{N\le k\le
2N-1}(k+1)b_{k+1}\left(\int_0^1t^kf_N(t)\,dt\right)\Phi_N\left
(\frac{k}{N}\right )z^k
\\ & =\sum_{N\le k\le 2N-1}(k+1)b_{k+1}\psi_{N}\left (\frac{k}{N}\right )\varphi_N\left (\frac{k}{N}\right )z^k
\\ & =\sum_{N\le k\le 2N-1}(k+1)b_{k+1} z^k.
\end{split}\end{equation}
\par Using the M. Riesz projection theorem and (\ref{eq:nh3}) we have
$$
\|F_2^N\|_{H^p}\leq C_p||W_N^{\Phi_N}\ast \hg(f_N)||_{H^p}\leq  C_p N^{1-\frac1p},
$$
valid for each $N$. Finally observing that for $n\in \mathbb{N}$,
\begin{equation*}
\Delta_n g'(z) =\sum_{k=2^n}^{2^{n+1}-1}(k+1)b_{k+1} z^k= F_2^{2^n}(z),
\end{equation*}
we obtain
$$
||\Delta_n g'||_{H^p}\leq C_p2^{n(1-\frac{1}{p})},
$$
and using part (iv) of Theorem \ref{th:mlip}, we conclude    $g\in \Lambda\left(p,\frac{1}{p}\right)$.
\end{pf}
\par\bigskip
\begin{Pf}{\em{ the necessity statement in Theorem \ref{th:dirichlet}}}:
\par
The proof is similar to that
of Theorem~\ref{th:hardy.nec}, hence, we shall omit some details.
 Let $p$,  $\alpha$ and  $g(z)=\sum_{k=0}^\infty b_kz^k\in\hol(\D)$ be as in the
 statement and assume that $\hg:\Dpa\to \Dpa$ is bounded. We consider the
 functions $\psi_{N, \alpha}$ and $\phi_{N, \alpha}=\frac{1}{\psi_{N, \alpha}}$ defined in Lemma \ref{le:1}.
By part (iv) of Lemma \ref{le:1}, for each $N=1,2,\dots$, there is a
$C^\infty$-function $\Phi_{N,\a}:\R\to\C$ with
$\supp\left(\Phi_{N,\a}\right)\subset \left(\frac12, 4 \right)$ such
that
\begin{equation}\label{eq:nh1a}
\Phi_{N,\a}(s)=\varphi_{N,\a}\left(s+\frac{1}{N}\right),\quad 1\le s\le 2,
\end{equation}
and for each $m\in\N$ there exists $C(m)$ (independent of $N$) such that
\begin{equation}\label{eq:nh2a}|\Phi^{(m)}_{N,\a}(s)|\le C(m)N^{2-\frac{2+\a}{p}},\quad
s\in\R,\quad N=1,2,\dots.\end{equation}
\par Since  $\a<3p-2$, the family of test functions
\begin{equation}\label{eq:testfuncdpa}
f_N(z)=f_{N,\a}(z)=\frac{1}{N^{3-\frac{2+\a}{p}}}\frac{1}{(1-a_Nz)^2},\quad z\in\D,
\end{equation}
forms a bounded set in $\Dpa$ (see \cite[Lemma 3.10]{Zhu}), and the hypothesis implies that
$$
\sup_{N}||\hg(f_N)||_{\Dpa}<\infty.
$$
This, together with the easily checked  identity $\hg(f)'=\hgp(zf)$,
gives
 $$
 \sup_{N}||\hgp(zf_N)||_{A^p_\alpha}=C<\infty.
 $$
Then part (iv) of  Theorem~\ref{th:cesaro} and (\ref{eq:nh2a}) imply that
\begin{equation}
\begin{split}\label{eq:nh3a}
||W_N^{\Phi_{N,\a}}\ast \hgp(zf_N)||_{A^p_\alpha} &
\le C_p A_{\Phi_{N,\a}} ||\hgp(zf_N)||_{A^p_\alpha}\le C_p N^{2-\frac{2+\a}{p}}.
\end{split}
\end{equation}
Moreover,
\begin{align*}
(W_N^{\Phi_{N,\a}}\ast& \hgp(zf_N))(z)= \\
&=\sum_{\frac{N}{2}\leq k\leq
4N}(k+1)(k+2)b_{k+2}\left(\int_0^1t^{k+1}f_N(t)\,dt\right)
\Phi_{N,\alpha}\left (\frac{k}{N}\right)z^k
\end{align*}
and, by (\ref{eq:nh1a}),
\begin{equation}
\begin{split}\label{eq:nh4a}
\sum_{k=N}^{2N-1}(k+1)(k+2)b_{k+2}&\left(\int_0^1t^{k+1}f_N(t)\,dt\right)\Phi_{N,\alpha}\left
(\frac{k}{N}\right )z^k
\\ & =\sum_{k=N}^{2N-1}(k+1)(k+2)b_{k+2} z^k.
\end{split}
\end{equation}
\par Consequently, using (\ref{eq:nh4a}), the M. Riesz projection theorem, Lemma \ref{le:eqno}, and (\ref{eq:nh3a}),
and setting $N=2^n, \,\,n\in \mathbb{N}$,
\begin{equation*}
\begin{split}
||\Delta_n g''||_{H^p} &=\left\|\sum_{k=2^n}^{2^{n+1}-1}(k+1)(k+2)b_{k+2} z^k\right\|_{H^p}\\
&\le C_p\left\| \sum_{k=2^{n-1}}^{2^{n+2}}(k+1)(k+2)b_{k+2}\left(\int_0^1t^{k+1}f_{2^n}(t)\,dt\right)
\Phi_{2^n,\alpha}\left (\frac{k}{2^n}\right )z^k\right\|_{H^p}\\
& =  C_p\left\|W_{2^n}^{\Phi_{2^n, \a}}\ast\hgp(zf_{2^n})\right\|_{H^p}\\
& \le  C_p 2^{n(\frac{1+\alpha}{p})}\left\|W_{2^n}^{\Phi_{2^n, \a}}\ast\hgp(zf_{2^n})\right\|_{A^p_{\a}}\\
&\le C_p 2^{n(2-\frac1p)},
\end{split}
\end{equation*}
and by part (v) of  Theorem~\ref{th:mlip},  we deduce that $g\in \Lambda\left(p,\frac{1}{p}\right)$.
\end{Pf}
\par
We note that the proof we have just finished remains valid for $p>1$
and $\alpha<3p-2$. \par\medskip
\begin{Pf}{\em{ the necessity statement in Theorem \ref{th:bergman}}}:
\par
Let $p, \a$ be as in the statement and assume
$\hg:A^p_{\a}\to A^p_{\a}$ is bounded.
\par
Since $\a<p-2$, then $\a+p<3p-2$. This together with the fact
that $A^p_\a=\mathcal{D}^p_{p+\a}$ gives the assertion  as a consequence of the
preceding proof.
\end{Pf}
\section{The sublinear Hilbert operator}\label{sublinear}
\par
Let us consider the following  space of analytic functions in $\D$
$$
A^1_{[0,1)}=\left\{f\in \hol(\D):\,\int_0^1|f(t)|\,dt<\infty\right\}.
$$
The well-known Fej\'er-Riesz inequality \cite{D} implies that
$H^1\subset A^1_{[0,1)}$. We remark also that an application of
H\"older's inequality yields
 \begin{equation}\label{eq:apaa1}
 A^p_\alpha\subset
 A^1_{[0,1)},\quad \text{if  $p>1$ and $-1<\alpha<p-2$,}
 \end{equation}
 an inclusion which is not longer true for $\alpha\ge p-2$.
\par
 Condition (\ref{eq:apaa1}) insures that $\hti$ is well defined on  $A^p_\alpha$
  for $p$ and $\alpha $ in that range of values.
 \par
 Now, we proceed to state some lemmas which will be needed for the proof Theorem~\ref{th:hti}.
 \begin{lemma}\label{le:hti1}
 \begin{itemize}
 \item[(i)] Assume that $0<p<\infty $. Then there exists a positive constant
 $C=C(p)$ such that
 $$\int_0^1M_\infty ^p(r,g)\,dr\le C\Vert g\Vert _
 {H^p}^p,\quad\text{for all $g\in \hol (\D )$}.$$
 \item[(ii)]
 Assume that $0<p<\infty$ and $\a>-1$. Then there exists a positive constant $C=C(p,\alpha)$ such that
 $$
 \int_0^1 M^p_\infty(r,f)(1-r)^{\a+1}\,dr\le C ||f||^p_{A^p_\alpha},\quad\text{for all $f\in \hol (\D
 )$}.
 $$
 \end{itemize}
 \end{lemma}
 \begin{pf}
 Part (i) follows taking $q=\infty $ and $\lambda =p$ in
 Theorem\,\@5.\,\@11 of \cite{D}.
 \par Now we proceed to prove part\,\@(ii).
Applying (i) to $g(z)=f(sz)$ ($0<s<1$) and making a change of
variables, we obtain
 $$
 \int_0^s M^p_\infty(r,f)\,dr \le Cs M_p^p(s,f)\quad 0\le r<1,
 $$
 Multiplying both sides of the last inequality by $(1-s)^\alpha $,
 integrating the resulting inequality, and applying Fubini's theorem
 yieds
\begin{equation*}
\int_0^1 M^p_\infty(r,f)(1-r)^{\a+1}\,dr=C\int_0^1(1-s)^\alpha
\int_0^sM_\infty ^p(r,f)\,dr\,ds\,\le \, C\Vert f\Vert _{A^p_\alpha
}^p.\end{equation*}
\end{pf}
 \begin{lemma}\label{le:hti2}
 Assume that $1<p<\infty$ and $p-2<\alpha$. Then there exists a constant $C=C(p,\alpha)$
 such that for any $f\in\hol(\D)$
 $$
 \int_0^1 M^p_\infty(r,f)(1-r)^{\a-p+1}\,dr\le C ||f||^p_{\Dpa}.
 $$
 \end{lemma}
 \begin{proof}
 The identity $f(z)=f(0)+\int_0^zf'(\zeta)\,d\zeta$, $z\in \D$,
 gives
 $$
 M_{\infty}^p(r, f)\leq C\left(|f(0)|^p+\left(\int_0^r M_{\infty}(t,  f')\,dt\right)^p\right),
 $$
 for some constant $C$. Since $\a-p+1>-1$ we have
\begin{align*}
 \int_0^1 M^p_\infty&(r,f) (1-r)^{\a-p+1}\,dr\\
&  \le C |f(0)|^p+C\int_0^1\left(\int_0^r M_\infty(t,f')\,dt\right)^p (1-r)^{\a-p+1}\,dr\\
& =  C |f(0)|^p+ C\int_0^1\left(\int_{1-r}^1 M_\infty(1-s,f')\,ds\right)^p (1-r)^{\a-p+1}\,dr.
\end{align*}
We now  use  the following  version of the  classical Hardy inequality
 \cite[p.~244-245]{HLP-88}: If $k>0$, $q>1$ and $h$ is a nonnegative
 function defined in $(0,\infty )$ then
 $$
 \int_0^\infty\left(\int_x^\infty h(t)dt\right)^q x^{k-1} dx
 \leq \left(\frac{q}{k}\right)^q\int_0^\infty h(x)^q x^{q+k-1}dx.
 $$
Taking $h\equiv 0$ in $[1,\infty)$, and making the change of variable $x=1-r$ in each side, the
inequality takes the form
\begin{equation}\label{eq:hardy}
\int_0^1\left(\int_{1-r}^1 h(t)\, dt \right)^q (1-r)^{k-1}\,dr
\leq  \left(\frac{q}{k}\right)^q \int_0^1  (h(1-r))^q (1-r)^{q+k-1}\,dr.
\end{equation}
Now apply this inequality to the function $h(s)=M_{\infty}(1-s, f')$ with $k=\a-p+2>0$ to obtain
\begin{align*}
\int_0^1  \left(\int_{1-r}^1 M_\infty(1-s,f')\,ds \right)^p&  (1-r)^{\alpha-p+1}\,dr\\
 &\le C \int_0^1 M^p_\infty(r,f') \,(1-r)^{\a+1}\,dr.
\end{align*}
Putting together the above and using  Lemma \ref{le:hti1} we find,
\begin{align*}
 \int_0^1 M^p_\infty(r,f) (1-r)^{\a-p+1}\,dr&
 \le C \left(|f(0)|^p+\int_0^1 M^p_\infty(r,f') \,(1-r)^{\a+1}\,dr,\right)\\
&\le C (|f(0)|^p+||f'||^p_{A^p_{\a}})\\
&=C ||f||^p_{\Dpa},
\end{align*}
and the proof is complete.
\end{proof}
\par
The first part of the following Lemma is a special case of
\cite[Theorem $2.1$]{LueInd85}, and the second part is an immediate
consequence of the first part.
\begin{lemma}\label{dual-Berg}
(i) If  $1<p<\infty$ and $\a>-1$, then  the dual of $A^p_\alpha $ can
be identified with  $A^{q}_{\beta}$ where $\frac{1}{p}+\frac{1}{q}=1$ and $\b$ is any number with $\b>-1$, under the pairing
\begin{equation}\label{pair-Berg}
\langle f, g\rangle_{A_{p,\alpha ,\beta }}=\int_\D
f(z)\overline{g(z)}(1-|z|^2)^{\frac{\a}{p}+\frac{\b}{q}}\,dA(z).
\end{equation}
(ii) If  $1<p<\infty$ and $\a>-1$, then  the dual of $\Dpa$ can
be identified with  $\Dqb$ where $\frac{1}{p}+\frac{1}{q}=1$ and $\b$ is any number with $\b>-1$, under the pairing
\begin{equation}\label{pair-Dir}
\langle f, g\rangle_{\mathcal {D}_{p,\alpha ,\beta
}}=f(0)\overline{g(0)}+\int_\D
f'(z)\overline{g'(z)}(1-|z|^2)^{\frac{\a}{p}+\frac{\b}{q}}\,dA(z).
\end{equation}
\end{lemma}
 \begin{Pf}{\em{ Theorem~\ref{th:hti}.}}
 \par (i)\, Recall that for $1<p<\infty$, the dual of $H^{p}$ can be identified with  $H^{q}$, $\frac{1}{p}+\frac{1}{q}=1$, under the $H^2$-pairing,
 $$
 \langle f, h \rangle_{H^2}=\lim_{r\to 1}\frac{1}{2\pi}\int_{0}^{2\pi}f(re^{i\theta})\overline{h(re^{i\theta})}\,d\theta,
 $$
 thus  it is enough to prove that there
 exists a constant $C>0$ such that
 \begin{equation}\label{eq:htihp1}
 \left|\lim_{r\to 1^-}\frac{1}{2\pi}\int_{0}^{2\pi}\hti(f)(re^{i\theta})\overline{h(re^{i\theta})}\,d\theta
 \right|\le C||f||_{H^{p}}||h||_{H^{q}}
 \end{equation}
  for any $f\in H^p$ and $g\in H^{q}$. Now, by Fubini's theorem
 \begin{align*}
  \frac{1}{2\pi}\int_{0}^{2\pi}\hti(f)(re^{i\theta})\overline{h(re^{i\theta})}\,d\theta=
  &
  \int_0^1|f(t)|\overline{\left(\frac{1}{2\pi}\int_0^{2\pi}\frac{h(re^{i\theta})}{1-tre^{-i\theta}}\,d\theta\right)}\,dt\\
  = &\int_0^1|f(t)| \overline{h(r^2t)} \,dt.
  \end{align*}
 Using H\"older's inequality and the Fej\'er-Riesz inequality we have
  \begin{equation*}
  \begin{split}
  \left|\frac{1}{2\pi}\int_{0}^{2\pi}\hti(f)(re^{i\theta})\overline{h(re^{i\theta})}\,d\theta\right|
 &\le \left(\int_0^1 |f(t)|^p\,dt\right)^{1/p}\left(\int_0^1 |h(r^2t)|^{q}\,dt\right)^{1/q}
 \\ & \le C ||f||_{H^{p}} M_{q}(r^2,h)
 \\ & \le   C||f||_{H^{p}}||h||_{H^{q}},
  \end{split}
  \end{equation*}
  which implies (\ref{eq:htihp1}) and finishes the proof of (i).
 \par \vspace{1em}
(ii) Using Lemma \ref{dual-Berg} we can choose
$$
\b= \frac{-\a q}{p}=\frac{-\a}{p-1}
$$
so that  the weight in the  pairing (\ref{pair-Berg}) is identically equal to 1, and we have
for  $f\in A^p_\alpha$ and $h\in A^{q}_{\b}$,
\begin{equation}
\begin{split}\label{eq:htiapa1}
\langle \hti(f), h\rangle_{A_{p,\alpha ,\beta }}&=
 \int_\D \hti(f)(z)\overline{h(z)}\,dA(z) \\
 & =  \int_\D \left(\int_0^1\frac{|f(t)|}{1-tz}\,dt\right)\overline{h(z)}\,dA(z)\\
 & =\int_0^1|f(t)| \overline{\left(\int_\D\frac{h(z)}{1-t\bar{z}}\,dA(z)\right)}\,dt\\
 & =2 \int_0^1|f(t)| \left(\int_0^1\overline{h(r^2t)}r\,dr\right)\,dt,
\end{split}
\end{equation}
so that
\begin{equation}\label{htiapa2}
|\langle \hti(f), h\rangle_{A_{p,\alpha ,\beta }}|\leq 2
\int_0^1|f(t)| G(t)\,dt,
\end{equation}
where  $G(t)=\int_0^1|h(r^2t)|r\,dr$.
Using H\"older's inequality we obtain,
\begin{equation*}
\begin{split}\label{htiapa10}
\int_0^1|f(t)| G(t)\,dt&=\int_0^1|f(t)|(1-t)^{\frac{\a+1}{p}} G(t)(1-t)^{-\frac{\a+1}{p}}\,dt\\
 &\mspace{-80mu} \le \left(\int_0^1|f(t)|^p(1-t)^{\a+1}\,dt\right)^{1/p}
 \left(\int_0^1|G(t)|^{q}(1-t)^{-\frac{q(\a+1)}{p}}\,dt\right)^{1/q}\\
 & \mspace{-80mu}\le\left(\int_0^1M_{\infty}^p(t, f)(1-t)^{\a+1}\,dt\right)^{1/p}\left(\int_0^1|G(t)|^{q}(1-t)^{-\frac{q(\a+1)}{p}}\,dt\right)^{1/q}\\
 &\mspace{-80mu}\leq
C\n{f}_{A^p_\a} \left(\int_0^1|G(t)|^q (1-t)^{-\frac{q(\alpha+1)}{p}}\,dt\right)^{1/q},
\end{split}
\end{equation*}
where in the last step we have used Lemma \ref{le:hti1}. Next we show that
\begin{equation}\label{eeee}
 \int_0^1|G(t)|^q(1-t)^{-\frac{q(\alpha+1)}{p}}\,dt\le C ||h||^q_{A^{q}_{\b}}.
\end{equation}
This together with (\ref{htiapa2}) will finish the proof. To show
(\ref{eeee}) observe first that if  $0<t<1/2$ then  $|h(r^2t)|\leq
M_{\infty}\left (\frac{1}{2}, h\right )$ for each $0<r<1$, thus
$$
\int_0^1|h(r^2t)|rdr\leq M_{\infty}\left(\frac{1}{2}, h\right ),
\quad 0<t<1/2,
$$
and we have
\begin{align*}
\int_0^{1/2}|G(t)|^q (1-t)^{-\frac{q(\a+1)}{p}}\,dt &=
\int_0^{1/2}\left(\int_0^1|h(r^2t)|rdr\right)^{q}(1-t)^{-\frac{q(\alpha+1)}{p}}\,dt\\
&\le C M^{q}_\infty(\frac{1}{2}, h)\\
&\le C  ||h||^{q}_{A^{q}_{\b}}.
\end{align*}
On the other hand,
$$
-\frac{q(\a+1)}{p}=\frac{-\a+p-2}{p-1}-1>-1,
$$
and making a change of variable we obtain $\int_0^1|h(r^2t)|rdr=
\frac{1}{2t}\int_0^t|h(s)|ds$ so,
\begin{equation*}
\begin{split}
\int_{1/2}^1|G(t)|^p(1-t)^{-\frac{q(\a+1)}{p}}\,dt &=
\int_{1/2}^1\left(\int_0^1|h(r^2t)|rdr\right)^{q}(1-t)^{-\frac{q(\alpha+1)}{p}}\,dt\\
& \mspace{-40mu}= \int_{1/2}^1\frac{1}{(2t)^q}\left(\int_0^t|h(s)|ds\right)^q(1-t)^{-\frac{q(\alpha+1)}{p}}\,dt\\
&\mspace{-40mu}\leq \int_{1/2}^1 \left(\int_0^t M_\infty(s,h)\,ds\right)^{q}(1-t)^{-\frac{q(\alpha+1)}{p}}\,dt\\
&\mspace{-40mu} \le  \int_{0}^1\left(\int_{1-t}^1 M_\infty(1-s,h)\,ds\right)^{q}(1-t)^{-\frac{q(\a+1)}{p}}\,dt\\
&\mspace{-40mu} \le C \int_{0}^1 M^{q}_\infty(t, h) (1-t)^{-\frac{q\a}{p}+1} \,dt\quad (\mbox{by}\,\, (\ref{eq:hardy})) \\
& \mspace{-40mu} \le C  ||h||^{q}_{A^{q}_{\b}},
\end{split}
\end{equation*}
where we have used  Lemma \ref{le:hti1} in the last step. Thus (\ref{eeee}) is proved and the
proof of (ii) is complete.
\par
\vspace{1em} (iii)\, {\bf{Case $\mathbf{\a=p-1}$}}. By Lemma
\ref{dual-Berg}  the dual of  $\Dp$ can be identified with $\Dq$,
$\frac{1}{p}+\frac{1}{q}=1$, taking $\alpha =p-1$ and $\beta =q-1$
in the relevant pairing in (\ref{pair-Dir}), so that the weight
becomes $(1-|z|^2)$. Thus for $f\in \Dp$ and $h\in \Dq$ we have by
Fubini's theorem
\begin{align*}
\langle \hti(f), h\rangle_{\mathcal{D}_{p,p-1,q-1}} &= \hti(f)(0)\overline{h(0)}+\int_\D \hti(f)'(z)\overline{h'(z)}\, (1-|z|^2)\,dA(z)\\
&= \hti(f)(0)\overline{h(0)}+\int_0^1|f(t)|\left(\int_{\D}\frac{t\overline{h'(z)}}{(1-tz)^2}(1-|z|^2)
dA(z)\right)dt,
\end{align*}
and  a routine calculation gives
$$
\int_{\D}\frac{t\overline{h'(z)}}{(1-tz)^2}(1-|z|^2)
dA(z)= \overline{h(t)}-\int_0^1\overline{h(rt)}\,dr.
$$
Now
$$
\left|\overline{h(t)}-\int_0^1\overline{h(rt)}\,dr\right|\leq 2M_{\infty}(t, h)
$$
therefore,
\begin{align*}
\left|\int_\D \hti(f)'(z)\overline{h'(z)}\, (1-|z|^2)\,dA(z)\right|&\leq
\int_0^1|f(t)|\left|\overline{h(t)}-\int_0^1\overline{h(rt)}\,dr\right|\,dt\\
&\mspace{-100mu} \leq 2\int_0^1 M_{\infty}(t, f)M_{\infty}(t, h)\,dt\\
&\mspace{-100mu}\leq 2\left(\int_0^1 M_{\infty}^p(t, f)\,dt\right)^{1/p} \left(\int_0^1 M_{\infty}^q(t, h)\,dt\right)^{1/q}\\
&\mspace{-100mu}\leq C\n{f}_{\Dp}\n{h}_{\Dq}
\end{align*}
where for the last inequality we have used Lemma \ref{le:hti2} twice with $\a=p-1$ and $\a=q-1$ in the two integrals respectively.
Moreover $|h(0)|\leq \n{h}_{\Dq}$ and
$$
|\hti(f)(0)|= \int_0^1|f(t)|\,dt  \leq \left(\int_0^1 M_{\infty}^p(t, f)\,dt\right)^{1/p}\leq
C\n{f}_{\Dp},
$$
and combining the above we obtain
$$
|\langle \hti(f), h\rangle_{\mathcal{D}_{p,p-1,q-1}}|\leq
C\n{f}_{\Dp}\n{h}_{\Dq}
$$
which completes the proof of this case.
 \par\vspace{1em} {\bf{Case
$\mathbf{p-2<\alpha<p-1}$}}. In this case the dual of $\Dpa$ can be
identified with $\mathcal{D}^{q}_{\b} $ with $\b=\frac{-\a q}{p}$.
The weight in the  pairing (\ref{pair-Dir}) is then identically
equal to $1$. Thus for $f\in \Dpa$ and $h\in \mathcal{D}^{q}_{\b}$
we have
$$
\langle \hti(f), h\rangle_{\mathcal{D}_{p,\alpha ,\beta
}}=\hti(f)(0)\overline{h(0)}+\int_\D
\hti(f)'(z)\overline{h'(z)}\,dA(z).
$$
 Now using Fubini's theorem and the reproducing formula
 $$
 h'(a)=\int_{\D}\frac{h'(z)}{(1-a\bar{z})^2}\,dA(z), \quad a\in \D,\,\, h\in \D^q_{\b},
 $$
 we find
\begin{align*}
\int_\D \hti(f)'(z)\overline{h'(z)}\,dA(z)&
=\int_0^1t|f(t)|\left(\int_{\D}\frac{\overline{h'(z)}}{(1-tz)^2}\,dA(z)\right)\,dt\\
&=\int_0^1t|f(t)|\overline{h'(t)}\,dt.
\end{align*}
We set  $s=-1+\frac{\a+1}{p}$ and use  H\"older's inequality to obtain
\begin{equation*}
\begin{split}
\left|\int_0^1t|f(t)|\overline{h'(t)}\,dt\right|&=
\left|\int_0^1t|f(t)|(1-t)^s\overline{h'(t)}(1-t)^{-s}\,dt\right|\\
&\mspace{-60mu}\leq \left(\int_0^1|f(t)|^p(1-t)^{ps}\,dt\right)^{\frac{1}{p}}
 \left(\int_0^1|h'(t)|^{q}(1-t)^{-qs}\,dt\right)^{\frac{1}{q}}.
\end{split}
\end{equation*}
By Lemma \ref{le:hti2} the first integral above is
\begin{equation*}
\begin{split}
\int_0^1|f(t)|^p(1-t)^{ps}\,dt&= \int_0^1|f(t)|^p(1-t)^{\a-p+1}\,dt\\
&\leq \int_0^1M_{\infty}^p(f, t)(1-t)^{\a-p+1}\,dt\\
&\leq C\n{f}_{\Dpa}^p ,
\end{split}
\end{equation*}
while the second integral by Lemma \ref{le:hti1} is
\begin{equation*}
\begin{split}
\int_0^1|h'(t)|^{q}(1-t)^{-qs}\,dt&=\int_0^1|h'(t)|^{q}(1-t)^{1-\frac{\a q}{p}}\,dt\\
&=\int_0^1|h'(t)|^{q}(1-t)^{\b+1}\,dt\\
&\leq \int_0^1M_{\infty}^q(h', t)(1-t)^{\b+1}\,dt\\
& \leq C\n{h'}_{A_{\b}^q}^q\\
&\leq C\n{h}_{\mathcal D_{\b}^q}^q .
\end{split}
\end{equation*}
Thus
$$
\left|\int_\D \hti(f)'(z)\overline{g'(z)}\,dA(z)\right|\leq
C\n{f}_{\Dpa} \n{g}_{\mathcal D_{\b}^q}.
$$
This together with the inequalities $|h(0)|\leq \n{h}_{\mathcal
D_{\b}^q}$ and
$$
|\hti(f)(0)|=\int_0^1|f(t)|\,dt\leq C\n{f}_{\Dpa}
$$
imply that
$$
|\langle \hti(f), h\rangle_{\mathcal{D}_{p,\alpha ,\beta }}|\leq
C\n{f}_{\Dpa}\n{h}_{\mathcal D^q_{\b}}
$$
and the proof is complete.
\end{Pf}
\section{Sufficient conditions}
\par
In this section we will prove the sufficient conditions for Theorems
1, 2(ii),  3, and 4. In order to do that we state first some needed
results.
\par
A nice result of Hardy-Littlewood \cite[Section $6.2$]{D} \cite[Theorem $7.5.1$]{Pabook} asserts that if
$0<p\le 2$  and $f(z)=\sum_{k=0}^\infty a_k z^k\in H^p$,\, then
 \begin{equation}\label{eq:hl1}
 K_p(f)=\sum_{k=0}^\infty (k+1)^{p-2}|a_k|^p\le C_p ||f||^p_{H^p}.
 \end{equation}
On the other hand, if $2\le p<\infty $ and $f(z)=\sum_{k=0}^\infty
a_k z^k\in \hol(D)$ satisfies that $K_p(f)<\infty$, then $f\in H^p$
and
 \begin{equation}\label{eq:hl2}
 ||f||^p_{H^p}\le C_p K_p(f).
 \end{equation}
The converse of each of these two statements is not true for a
general power series $f(z)=\sum_{k=0}^\infty a_k z^k\in \hol(D)$ and
for arbitrary indices $p\ne 2$. If however we restrict to the class
of power series with non-negative decreasing  coefficients then we
have the following result (see \cite{HL-31}, \cite[Chapter XII,
Lemma $6.6$]{Zy}, \cite[$7.5.9$]{Pabook} and \cite{Padec}).
 \begin{other}\label{th:decreasing}
 Assume that $1\le p<\infty$  and  $f(z)=\sum_{k=0}^\infty a_k z^k\in \hol(\D)$ where $\{a_n\}$ is a
 sequence of positive numbers which decreases to zero.
 Then the following assertions are equivalent:
 \par (i)\, $f\in H^p$.
 \par (ii)\, $f\in \Dp$.
 \par (iii)\, $K_p(f)<\infty$.
 \par Furthermore,
 $$
 ||f||^p_{H^p}\asymp ||f||_{\Dp}^p\asymp K_p(f).
 $$
 \end{other}
\par
 The following decomposition theorem can be found in \cite[Theorem $2.1$]{MP} and \cite[$7.5.8$]{Pabook}.
\begin{other}\label{th:dec}
\par (i)\, Assume  $0<p<\infty$, $1<q<\infty$ and  $0<\a<\infty$. Then,
$$
\int_0^1 (1-r)^{p\alpha-1}M_q^p(r,f)\asymp  |f(0)|^p+\sum_{n=0}^\infty 2^{-n(p\alpha)}||\Delta_nf||^p_{H^q},
$$
for all $f\in\hol(\D)$.
\par
(ii)\, In particular,
if $p>1$ and $\beta>-1$,
$$||f||^p_{A^p_\beta}\asymp |f(0)|^p+\sum_{n=0}^\infty 2^{-n(\beta+1)}||\Delta_nf||^p_{H^p}$$
for all $f\in\hol(\D)$.
\end{other}
\par
The following lemma can be found in \cite[$7.3.5$]{Pabook}, in a
slightly different form. The proof suggested there  can be applied
to obtain it in the form we need it. \par
\begin{lemma} \label{le:mult}
Suppose $0<p<\infty$ and   $\g\in \R$. For  $f(z)=\sum_{k=0}^\infty
a_k z^k\in \hol(\D)$ let $F(z)=\sum_{k=0}^\infty (k+1)^{\g}a_k z^k$.
Then
$$
\n{\Delta_n F}_{H^p}\asymp 2^{n\g}\n{\Delta_n f}_{H^p} .
$$
\end{lemma}
\par
\begin{lemma}\label{le:3}
Suppose that $1<p<\infty$. There exists a constant $C=C(p)>0$ such
that if $f\in H^1$, $g(z)=\sum_{k=0}^\infty c_k z^k\in \hol(\D),$
and we set
$h(z)=\sum_{k=0}^{\infty}c_k\left(\int_0^1t^{k+1}f(t)\,dt\right)z^k$
then
\begin{equation*}
\|\Delta_n h\|_{H^p} \le C\left(\int_0^1
t^{2^{n-2}+1}|f(t)|\,dt\right)\| \Delta_ng\|_{H^p},\quad n\ge 3.
\end{equation*}
\end{lemma}
\begin{proof} For  each $n=1,2, \dots$, define
$$
\Upsilon_n(s)=\int_0^1 t^{2^n s+1}f(t)\,dt,\quad s\ge 0.
$$
Clearly, $\Upsilon_n$ is a $C^\infty(0,\infty)$-function and
\begin{equation}\label{eq:up1}
\left| \Upsilon_n(s)\right|\le \int_0^1 t^{2^{n-2}+1}|f(t)|\,dt,\quad s\ge \frac12.
\end{equation}
\par Furthermore, since
$$
 \sup_{0<x<1}\left(\log
\frac{1}{x}\right)^2 x^{1/2}=C(2)<\infty,
$$
we have
\begin{equation}
\begin{split}\label{eq:up2}
\left| \Upsilon''_n(s)\right| &
\le \int_0^1\left[\left(\log\frac{1}{t^{2^n}}\right)^2 t^{2^{n-1}}\right]\,t^{2^ns+1-2^{n-1}}|f(t)|\,dt
\\  & \le C(2)\int_0^1 t^{2^ns+1-2^{n-1}}|f(t)|\,dt
\\  & \le C(2) \int_0^1 t^{2^{n-2}+1}|f(t)|\,dt, \quad s\ge \frac34.
\end{split}
\end{equation}
\par
 Then, using (\ref{eq:up1}) and (\ref{eq:up2}), for each $n=1,2,\dots$ we can take
a function $\Phi_n\in C^\infty(\R)$ with
$\supp(\Phi_n)\in\left(\frac34, 4\right)$, and such that
$$
\Phi_n(s)= \Upsilon_n(s),\quad s\in\,[1,2],
$$
and
$$
A_{\Phi_n}=\max_{s\in\R}|\Phi_n(s)|+\max_{s\in\R}|\Phi_n''(s)|\le C\int_0^1 t^{2^{n-2}+1}|f(t)|\,dt.
$$
We can then write
\begin{align*}
\Delta_nh(z)&=\sum_{k\in I(n)}c_k\left(\int_0^1t^{k+1}f(t)\,dt\right)z^k\\
&=\sum_{k\in I(n)}c_k\Phi_n\left (\frac{k}{2^n}\right )z^k\\
&=W_{2^n}^{\Phi_n}\ast\Delta_n g(z).
\end{align*}
So by using
part (iii) of Theorem~\ref{th:cesaro}, we have
\begin{equation*}
\begin{split}
\| \Delta_n h\|_{H^p}
&= \|W_{2^n}^{\Phi_{n}}\ast \Delta_n g\|_{H^p}\\
& \leq C_pA_{\Phi_n}\| \Delta_n g\|_{H^p}\\
&\leq C \left(\int_0^1 t^{2^{n-2}+1}|f(t)|\,dt\right)\| \Delta_n
g\|_{H^p}.
\end{split}
\end{equation*}
\end{proof}
\par
We shall need  several lemmas. The first one can be found in
\cite[p.4]{Padec}
\par
 \begin{otherl}\label{le:A}
 Assume that $1<p<\infty$ and $\lambda=\left\{\lambda_n\right\}_{n=0}^\infty$
 is a monotone sequence of non negative numbers. Let $(\lambda g)(z)=\sum_{n=0}^\infty \lambda_nb_n z^n$, where $g(z)=\sum_{n=0}^\infty b_n z^n$.
 Then:
 \par $(a)$\, If $\left\{\lambda_n\right\}_{n=0}^\infty$ is nondecreasing, there is $C>0$ such that
 $$
 C^{-1}\lambda_{2^{n-1}} ||\Delta_n g||^p_{H^p}\le ||\Delta_n \lambda g||^p_{H^p}\le C\lambda_{2^{n}} ||\Delta_n g||^p_{H^p}.
 $$
 \par $(b)$\, If $\left\{\lambda_n\right\}_{n=0}^\infty$ is nonincreasing, there is $C>0$ such that
 $$
 C^{-1}\lambda_{2^{n}} ||\Delta_n g||^p_{H^p}\le ||\Delta_n \lambda g||^p_{H^p}\le C\lambda_{2^{n-1}} ||\Delta_n g||^p_{H^p}.
 $$
 \end{otherl}
\par
\begin{lemma}\label{le:htidec}
\par
(i) Assume $1<p<\infty$,  $-1<\a<\infty$,  and $f\in \Dpa$. Then
 $$
 ||\hti(f)||^p_{\Dpa}\asymp |\hti(f)(0)|^p+\sum_{j=1}^\infty (j+1)^{2p-3-\alpha}
 \left(\int_0^1 t^{j+1}|f(t)|\,dt\right)^p.
 $$
\par $(ii)$ Assume  $1<p<\infty$,  $-1<\alpha<p-2$, and $f\in A^p_\alpha$. Then
 $$
 ||\hti(f)||^p_{A^p_\a}\asymp |\hti(f)(0)|^p+\sum_{j=1}^\infty (j+1)^{p-3-\alpha}
 \left(\int_0^1 t^{j}|f(t)|\,dt\right)^p.
 $$
 \end{lemma}\par
 \begin{proof}
 Set $r_n=1-\frac{1}{2^{n}}$. Applying \cite[Lemma 3.1]{MP} to the function $h(z)=\frac{1}{1-z}=\sum_{k=0}^{\infty}z^k$, we deduce that
 \begin{equation}\label{j1}
 \n{\Delta_nh}^p_{H^p}\asymp \int_{-\pi}^\pi\frac{1}{|1-r_ne^{it}|^p}\,dt\asymp \frac{1}{(1-r_n)^{p-1}}\asymp 2^{n(p-1)}.
 \end{equation}
\par Now, we shall prove $(i)$.
By Theorem~\ref{th:dec}~(ii) we have
 \begin{align*}
 \n{\hti(f)}^p_{\Dpa}&=|\hti(f)(0)|^p+\n{\hti(f)'}^p_{A^p_\a}\\
 &\asymp |\hti(f)(0)|^p+|\hti(f)'(0)|^p+\sum_{n=0}^\infty 2^{-n(\a+1)}\|\Delta_n\hti(f)'\|^p_{H^p}\\
 &\asymp |\hti(f)(0)|^p+\sum_{n=0}^\infty 2^{-n(\a+1)}\|\Delta_n\hti(f)'\|^p_{H^p},
 \end{align*}
 where we have taken into account that
 $$
 \hti(f)'(0)=\int_0^1t|f(t)|dt
 \asymp\ \int_0^1|f(t)|dt=\hti(f)(0).
 $$
We then apply Lemma \ref{le:mult} and subsequently  Lemma \ref{le:A}~(b) with the
 nonincreasing  sequence $\int_0^1 t^{j+1}|f(t)|dt$ and (\ref{j1}) to
 obtain
 \begin{align*}
 \sum_{n=0}^\infty 2^{-n(\a+1)}\|\Delta_n\hti(f)'\|^p_{H^p} &= \sum_{n=0}^\infty 2^{-n(\a+1)}\left\|\sum_{j\in I(n)}(j+1)\left(\int_0^1 t^{j+1}|f(t)|\,dt\right)z^j\right\|^p_{H^p}\\
 &\asymp \sum_{n=0}^\infty 2^{-n(\a+1-p)}\left\|\sum_{j\in I(n)}\left(\int_0^1 t^{j+1}|f(t)|\,dt\right)z^j\right\|^p_{H^p}\\
 & \lesssim   \sum_{n=0}^\infty 2^{-n(\alpha+1-p)}
 \left(\int_0^1 t^{2^{n-1}+1}|f(t)|\,dt\right)^p
 \left\|\sum_{j\in I(n)}z^j\right\|^p_{H^p}\\
 & \asymp \sum_{n=0}^\infty 2^{n(2p-2-\alpha)}\left(\int_0^1 t^{2^{n-1}+1}|f(t)|\,dt\right)^p\\
 & \asymp \sum_{j=1}^\infty (j+1)^{(2p-3-\alpha)}\left(\int_0^1 t^{j+1}|f(t)|\,dt\right)^p.
 \end{align*}
\par
Analogously, it can be proved that
$$
\sum_{n=0}^\infty 2^{-n(\a+1)}\|\Delta_n\hti(f)'\|^p_{H^p}
\gtrsim \sum_{j=1}^\infty (j+1)^{(2p-3-\alpha)}\left(\int_0^1 t^{j+1}|f(t)|\,dt\right)^p.
$$
and the assertion of (i) follows.
\par The proof of (ii) is similar and is omitted.
\end{proof}
\par
We are now ready to prove the sufficient conditions.\par
\vspace{0.5cm}
\begin{Pf}{\em{the sufficiency statement in Theorem \ref{th:dirichlet}}}:
\par
Let $p, \a$ be as in the statement and assume $g\in \Lambda\left(p, \frac{1}{p}\right)$.
For  $f\in\Dpa$,  bearing in mind that
$\hg(f)'(z)=\hgp(zf)$ and using Theorem~\ref{th:dec}~(ii), we obtain
\begin{equation*}
\begin{split}\label{eq:sd1}
\n{\hg(f)}^p_{\Dpa}  &=|\hg(f)(0)|^p+||\hgp(zf)||^p_{A^p_\a}\\
&=|\hg(f)(0)|^p+|\hgp(zf)(0)|^p+\sum_{n=0}^\infty 2^{-n(\a+1)}\n{\Delta_n \hgp(zf)}^p_{H^p} .
\end{split}
\end{equation*}
Now,
\begin{equation*}
\begin{split}
|\hg(f)(0)|^p+|\hgp(zf)(0)|^p&\leq (|g'(0)|^p+|g''(0)|^p)\left(\int_0^1|f(t)|\,dt\right)^p\\
&\leq C(g)\int_0^1|f(t)|^p\,dt\\
&\leq C(g)\int_0^1 M_{\infty}^p(t, f)(1-t)^{\alpha -p+1}\,dt  \\
&\leq C(g, p, \a)\n{f}_{\Dpa}^p
\end{split}
\end{equation*}
where in the last step we have used Lemma \ref{le:hti2} and the
observation that since $p-2<\a\leq p-1$ we have $-1<\a-p+1\leq 0$.
\par On the other hand, if we write
$g''(z)=\sum_{k=0}^{\infty}c_kz^k$ then
$$
\hgp(zf)(z)=\sum_{k=0}^{\infty} c_k\left(\int_0^1t^{k+1}f(t)\,dt\right)z^k
$$
and we can apply Lemma \ref{le:3} and part (v) of Theorem~\ref{th:mlip} to obtain
\begin{align*}
\n{\Delta_n \hgp(zf)}^p_{H^p}&\leq C\left(\int_0^1t^{2^{n-2}+1}|f(t)|dt\right)^p
\n{\Delta_n g''}_{H^p}^p\\
&\leq C \left(\int_0^1t^{2^{n-2}+1}|f(t)|dt\right)^p2^{pn(2-\frac{1}{p})}
\end{align*}
for $n\geq 3$. Thus
\begin{align*}
\sum_{n=3}^\infty 2^{-n(\a+1)}\n{\Delta_n &\hgp(zf)}^p_{H^p}
\leq  C \sum_{n=3}^\infty 2^{n(2p-2-\alpha)}\left(\int_0^1t^{2^{n-2}+1}|f(t)|\,dt\right)^p\\
&\asymp C\sum_{n=0}^\infty 2^{(n+1)(2p-2-\alpha)}\left(\int_0^1t^{2^{n+1}+1}|f(t)|\,dt\right)^p ;
\end{align*}
Now it is easy to see that
$$
2^{n}\left(\int_0^1t^{2^{n+1}+1}|f(t)|dt\right)^p\leq \sum_{j\in I(n)}\left(\int_0^1t^{j+1}|f(t)|dt\right)^p ,
$$
and we can continue the above estimate as follows
\begin{align*}
\mspace{20mu} & \leq C\sum_{n=0}^\infty 2^{(n+1)(2p-3-\alpha)}\sum_{j\in I(n)}\left(\int_0^1 t^{j+1}|f(t)|\,dt\right)^p\\
&\asymp C\sum_{j=1}^\infty (j+1)^{(2p-3-\alpha)}\left(\int_0^1t^{j+1}|f(t)|\,dt\right)^p\\
&\leq C\left(|\hti(f)(0)|^p+\sum_{j=1}^\infty (j+1)^{(2p-3-\alpha)}\left(\int_0^1t^{j+1}|f(t)|\,dt\right)^p\right)\\
&\asymp \n{\hti(f)}_{\Dpa}^p  \\
&\le C\n{f}_{\Dpa}^p,
\end{align*}
where we have used  Lemma~\ref{le:htidec}~(i)   and Theorem~\ref{th:hti}~(iii).
This  together with the  inequality for $|\hg(f)(0)|^p+|\hgp(zf)(0)|^p$
finishes the proof.
\end{Pf}
\par
\vspace{1em}
\begin{Pf}{\em{the sufficiency statement in Theorem \ref{th:bergman}.}}
\par
 Let $p, \a$ be as in the statement and let
$g(z)=\sum_{k=0}^\infty b_k z^k$ be the power series for $g$. For   $f\in A^p_\a$,  from
Theorem \ref{th:dec} we have
\begin{equation}
\begin{split}\label{eq:sb1}
\n{\hg(f)}^p_{A^p_\a}   \asymp & |\hg(f)(0)|^p+\sum_{n=0}^\infty
2^{-n(\alpha+1)}\n{\Delta_n\hg(f)}^p_{H^p} .
\end{split}
\end{equation}
Now,
\begin{align*}
|\hg(f)(0)|&\leq |g'(0)|\int_0^1|f(t)|\,dt\\
&= |g'(0)|\int_0^1|f(t)|(1-t)^{\frac{\a+1}{p}} (1-t)^{-\frac{\a+1}{p}}\,dt\\
\intertext{and by  H\"older's inequality,}
&\leq C(g,p,\a)\left(\int_0^1 M_{\infty}^p(t, f)(1-t)^{\a+1}\,dt\right)^{1/p}\\
&\leq C(g,p,\a)\n{f}_{A^p_\a},
\end{align*}
where the last inequality is from Lemma \ref{le:hti1}.
\par
Now write $g'(z)=\sum_{k=0}^{\infty}c_kz^k$,
then
$$
\hg(f)(z)=\sum_{k=0}^{\infty}\left(c_k\int_0^1t^kf(t)\,dt\right)z^k.
$$
Now  Lemma~\ref{le:3}
remains valid if we replace the power $t^{k+1}$ appearing in the definition of the function $h$ in statement of
the Lemma by $t^k$,  and the power $t^{2^{n-2}+1}$ in the conclusion by $t^{2^{n-2}}$. This variation
can be proved in the same way as the original version. Applying the Lemma in this new form and using the assumption for $g$ we find
\begin{align*}
\n{\Delta_n\hg(f)}_{H^p}^p&\leq C\left(\int_0^1t^{2^{n-2}}|f(t)|\,dt\right)^p\n{\Delta_ng'}_{H^p}^p\\
&\leq C\left(\int_0^1t^{2^{n-2}}|f(t)|\,dt\right)^p
2^{pn(1-\frac{1}{p})}.
\end{align*}
Now, the proof can be completed as the previous one using Theorem
\ref{th:dec}~(ii), Lemma~\ref{le:htidec}~(ii) and
Theorem~\ref{th:hti}~(ii). Namely, \par
\begin{align*}
\sum_{n=3}^\infty 2^{-n(\a+1)}&\n{\Delta_n \hg(f)}^p_{H^p}\leq  C \sum_{n=3}^\infty 2^{n(p-2-\alpha)}\left(\int_0^1t^{2^{n-2}}|f(t)|\,dt\right)^p\\
&\asymp C\sum_{n=0}^\infty 2^{(n+1)(p-2-\alpha)}\left(\int_0^1t^{2^{n+1}}|f(t)|\,dt\right)^p \\
 & \leq C\sum_{n=0}^\infty 2^{(n+1)(p-3-\alpha)}\sum_{j\in I(n)}\left(\int_0^1 t^{j}|f(t)|\,dt\right)^p\\
&\asymp C\sum_{j=1}^\infty (j+1)^{(p-3-\alpha)}\left(\int_0^1t^{j}|f(t)|\,dt\right)^p\\
&\leq C\left(|\hti(f)(0)|^p+\sum_{j=1}^\infty (j+1)^{(p-3-\a)}\left(\int_0^1t^{j}|f(t)|\,dt\right)^p\right)\\
&\asymp \n{\hti(f)}_{A^p_\a}^p  \\
&\leq C\n{f}_{A^p_\a}^p
\end{align*}
and  together with the inequality $|\hg(f)(0)|\leq  C(g,p,\a)\n{f}_{A^p_\a} $  this
finishes the proof.
 \end{Pf}
 \par\bigskip
 If $p>1$ and $f\in H^p$ then  $\hti(f)(z)=\sum_{j=0}^\infty\left(\int_0^1t^{j}|f(t)|\,dt\right)z^j$
is  analytic  in $\D$ and has nonnegative  Taylor coefficients
decreasing to zero. Thus   Theorem \ref{th:decreasing} implies that
 \begin{equation}\label{eq:equivhpdp}
 ||\hti(f)||^p_{\Dp}\asymp ||\hti(f)||^p_{H^p}\asymp\sum_{j=0}^\infty (j+1)^{p-2}\left(\int_0^1t^{j}|f(t)|\,dt\right)^p.
 \end{equation}
\par
\begin{Pf}{\em{the sufficiency statement in Theorem \ref{th:hardy1p2}.}}
Assume that $1<p\le 2$ and $g(z)=\sum_{k=0}^\infty b_k z^k\in
\Lambda \left(p,\frac{1}{p}\right)$. Take $f\in H^p$. Since
$\Dp\subset H^p$ with domination in the norms, by the proof of
Theorem~\ref{th:dirichlet} with $\alpha=p-1$, (\ref{eq:equivhpdp})
and Theorem \ref{th:hti} (i) we obtain
\begin{equation*}
\begin{split}\label{eq:hps1n}
||\hg(f)||^p_{H^p} \le C ||\hg(f)||^p_{\Dp} \le C
||\hti(f)||^p_{\Dp}\asymp ||\hti(f)||^p_{H^p}\le C||f||^p_{H^p}.
\end{split}
\end{equation*}
Hence $\hg:\,H^p\to H^p$ is bounded. This finishes the proof.
 \end{Pf}
\par\medskip
  \begin{Pf}{\em{Theorem \ref{th:hardy2pinfty}~(ii).}}
\par Let  $2<p<\infty$ and $g\in \Lambda (q,\frac{1}{q})$ for some
$q$ with $1<q<p$. Let $f\in H^p$. Applying \cite[Corollary
$3.1$]{MP} to the analytic function $\hg(f)$ we have,
$$
\n{\hg(f)}^p_{H^p}\le C\left(|\hg(f)(0)|^p+\int_0^1(1-r)^{p\left(1-\frac{1}{q}\right)}M^p_q(r,\hg(f)')\,dr\right)
$$
where $C=C(p,q)$ is an absolute constant. By Theorem~\ref{th:dec}~(i), applied here with $\a=1-\frac{1}{q}+\frac{1}{p}$ we further have
\begin{equation}\label{asymp}
\begin{split}
\int_0^1(1-r)^{p\left(1-\frac{1}{q}\right)}&M^p_q(r,\hg(f)')\,dr\\
&\asymp
|\hg(f)'(0)|^p+\sum_{n=0}^{\infty}2^{-n(p-\frac{p}{q}+1)}\n{\Delta_n\hg(f)'}_{H^q}^p.
\end{split}
\end{equation}
Now for the constant terms of the two relations above it is easy to see, using H\"older's
inequality and the Fejer-Riesz inequality that
\begin{equation}\label{con}
|\hg(f)(0)|^p+|\hg(f)'(0)|^p\leq C(g, p)\n{f}_{H^p}^p.
\end{equation}
To estimate the sum in  (\ref{asymp})  write $g''(z)=\sum_{k=0}^{\infty}c_kz^k$ so that
$$
\hg(f)'(z)= \hgp(zf)(z)=\sum_{k=0}^{\infty}\left(c_k\int_0^1t^{k+1}f(t)\,dt\right)z^k,
$$
and use  Lemma~\ref{le:3} and Theorem~\ref{th:mlip}~(v) to obtain
\begin{align*}
\sum_{n=3}^{\infty} 2^{-n(p(1-\frac{1}{q})+1)} &\n{\Delta_n\hg(f)'}_{H^q}^p\\
&\leq C\sum_{n=3}^{\infty}2^{-n(p-\frac{p}{q}+1)}\left(\int_0^1t^{2^{n-2}+1}|f(t)|\,dt\right)^p
\n{\Delta_n g''}^p_{H^q}\\
&\leq C\sum_{n=3}^{\infty}2^{n(p-1)}\left(\int_0^1t^{2^{n-2}+1}|f(t)|\,dt\right)^p\\
&\leq  C\sum_{n=0}^{\infty}2^{(n+1)(p-1)}\left(\int_0^1t^{2^{n+1}+1}|f(t)|\,dt\right)^p\\
&\leq C\sum_{n=0}^{\infty}2^{(n+1)(p-2)}\sum_{j\in I(n)}\left(\int_0^1t^{j+1}|f(t)|\,dt\right)^p\\
&\asymp C\sum_{j=0}^{\infty}(j+1)^{p-2}\left(\int_0^1t^{j+1}|f(t)|\,dt\right)^p\\
&\asymp \n{\hti(f)}_{H^p}^p \\
&\leq C \n{f}_{H^p}^p ,
\end{align*}
where in the  last two lines we have used (\ref{eq:equivhpdp}) and Theorem~\ref{th:hti}~(ii).
This and (\ref{con}) finish the proof.
\end{Pf}
\section{Compactness}\label{compactness}
Let us recall that an operator $T$ acting on a Banach space $X$ is
compact if any bounded sequence $\{f_k\}$ of elements of $X$ has a
subsequence $\{f_{k_i}\}$ such that $T(f_{k_i})$ converges in $X$.
For the generalized Hilbert operator  $\hg$ acting on the
appropriate spaces we have.
\begin{theorem}\label{th:hardyc}
Suppose that $1<p<\infty$ and  $g\in \hol(\D)$, then
\par (i)\, If $\hg:H^p\to H^p$ is compact  then $g\in \lambda(p,\frac{1}{p})$.
\par (ii)\, If $1<p\le 2$ and $g\in \lambda(p,\frac{1}{p})$, then $\hg:H^p\to H^p$ is compact.
\par (iii)\, If $2<p<\infty$ and $g\in \lambda(q,\frac{1}{q})$ for some $1<q<p$, then $\hg:H^p\to H^p$ is compact.
\end{theorem}
\begin{theorem}\label{th:bergmanc}
Suppose that $1<p<\infty$, $-1<\alpha<p-2$ and  $g\in \hol(\D)$. Then
 $\hg:A^p_\alpha\to A^p_\alpha$ is compact  if
 and only if $g\in \lambda\left(p,\frac{1}{p}\right)$.
\end{theorem}
\begin{theorem}\label{th:dirichletc}
Suppose that $1<p<\infty$, $p-2<\alpha\le p-1$ and  $g\in \hol(\D)$. Then
 $\hg:\Dpa\to \Dpa$ is compact  if
 and only if $g\in \lambda\left(p,\frac{1}{p}\right)$.
\end{theorem}
\par We shall use the following lemma.
\begin{lemma}\label{le:c1}
Suppose that $1<p<\infty$ and let $X$ be either $H^p$, or
$A^p_\alpha $ for some $\alpha $ with $-1<\alpha <p-2$, or $\mathcal
D^p_\alpha $ for some $\alpha $ with $p-2<\alpha \le p-1$. Let
$\{f_k\} _{k=1}^\infty $ be a sequence in $X$
 satisfying
$\sup_{k}\n{f_k}_{X}=K<\infty$ and $f_k\to 0$, as $k\to\infty $,
uniformly on compact subsets of\, $\D$. Then:
\begin{itemize}\item[(i)]
$\lim_{k\to\infty}\int_0^1 |f_k(t)|\,dt=0.$
\item[(ii)] For every $g\in \hol (\D )$ we have
$$\mathcal H_g(f_k)\to 0,\,\,\,\text{ as $k\to\infty $, uniformly on
compact subsets of \,$\D $.}$$
\end{itemize}
\end{lemma}
\begin{pf} Let's start with the proof of (i).
Let $q$ be the exponent conjugate to $p$, that is,
$\frac{1}{p}+\frac{1}{q}=1$. Take $\ep>0$. \par Suppose first that
$X=H^p$. Take $r_0\in (0,1)$ such that $(1-r_0)^{1/q}<\ep $. By the
hypothesis  there exists $k_0\in\N$ such that
$$
|f_k(z)|<\ep,\,\,\quad\text{if  $k\ge k_0$ and $\vert z\vert \le
r_0$}.
$$
Then, using H\"{o}lder's inequality and part\,\@(i) of
Lemma\,\@\ref{le:hti1}, we see that for $k\ge k_0$, we have
\begin{eqnarray*}
\int_0^1\vert f_k(t)\vert \,dt & \le & \ep + \int_{r_0}^1M_\infty
(t, f_k)\,dt
\\ & \le & \ep + \left ( \int_{r_0}^1M_\infty ^p
(t, f_k)\,dt\right )^{1/p}(1-r_0)^{1/q}
\\ & \le & \ep +CK\ep = C^\prime \ep .
\end{eqnarray*} Thus (i) holds in this case.
\par Similarly, if $X=A^p_\alpha $ with $-1<\alpha <p-2$,
take  $r_0\in (0,1)$ such that $(1-r_0)^{\frac{p-\a-2}{p}}<\ep$.
There exists $k_0\in\N$ such that
$$
|f_k(z)|<\ep,\,\,\quad\text{if  $k\ge k_0$ and $\vert z\vert \le
r_0$}.
$$
Then, using H\"{o}lder's inequality and part\,\@(ii) of
Lemma\,\@\ref{le:hti1}, we obtain, for $k\ge k_0$,
\begin{eqnarray*}
\int_0^1 |f_k(t)|\,dt & \le & \ep+\int_{r_0}^1 |f_k(t)|\,dt\\
& \le & \ep+\left(\int_{r_0}^1
M_{\infty}^p(t,f_k)(1-t)^{\a+1}\,dt\right)^{\frac{1}{p}}
\left(\int_{r_0}^1 (1-t)^{-(\a+1)\frac{q}{p}}\,dt\right)^{\frac{1}{q}}\\
&\leq & \ep+K\frac{p}{p-\a-2}(1-r_0)^{\frac{p-\a-2}{p}}\leq
C^\prime\ep.
\end{eqnarray*}
So, we see that (i) holds in this case too.
\par Finally, suppose that $X=\mathcal D^p_\alpha $ for a certain $\alpha $ with  $p-2<\alpha \le
p-1$. Since $\alpha -p\le -1$, we have that $\mathcal D^p_\alpha
\subset A^p_\beta $ for all $\beta >-1$. Take and fix $\beta $ with
$-1<\beta <p-2$. We have $X\subset A^p_\beta $ and then, using the
hypothesis and the closed graph theorem, we deduce that
$\sup_{k}\Vert f\Vert _{A^p_\beta }<\infty $ and then the result in
this case follows from the preceding one.
\par\medskip
Part\,\@(ii) follows easily from part\,\@(i). Indeed, if $g\in \hol
(\D )$ and $\vert z\vert \le r<1$, we have
$$\vert \mathcal H_g(f_k)(z)\vert =\left \vert
\int_0^1f_k(t)g^\prime (tz)\,dt\right \vert \le M_\infty (r,g^\prime
)\int _0^1\vert f_k(t)\vert \,dt. $$ Thus (ii) holds.
\end{pf}
\par\medskip Now the following result follows easily.
\begin{lemma}\label{le:c12} Suppose that $1<p<\infty$ and let $X$ be either $H^p$, or
$A^p_\alpha $ for some $\alpha $ with $-1<\alpha <p-2$, or $\mathcal
D^p_\alpha $ for some $\alpha $ with $p-2<\alpha \le p-1$. For a
function $g\in \hol(\D)$ the following conditions are
equivalent:\begin{itemize}\item [(i)] $\hg: X\to X$ is compact.
\item[(ii)] If  $\{f_k\}_{k=1}^\infty $ is a sequence in $X$ such that
 \begin{equation}\label{eq:com1}
\sup_{k}||f_k||_X=K <\infty
 \end{equation}
 and
 \begin{equation}\label{eq:com2}
f_k\to 0, \quad\text{as $k\to\infty $, uniformly on compact subsets
of $\D$},
\end{equation}
then  $\lim_{k\to\infty}||\mathcal H_g(f_k)||_{X}=0$.
\end{itemize}
\end{lemma}
\par
\begin{Pf}{\em{Theorem \ref{th:dirichletc}}}
 Assume first that  $\hg:\Dpa\to \Dpa$ is compact. Since the family of test functions
$$
f_{N,\alpha}(z)=\frac{1}{N^{3-\frac{2+\a}{p}}}\frac{1}{(1-a_Nz)^2},\quad z\in\D
$$
considered in (\ref{eq:testfuncdpa}) satisfies (\ref{eq:com1}) and (\ref{eq:com2}), we have
$$
\lim_{N\to\infty}||\mathcal H_g(f_{N,\alpha})||_{\Dpa}=0.
$$
Next, scrutinizing the proof  of Theorem~\ref{th:dirichlet}
(necessity part), we see that the quantity $||\mathcal
H_g(f_{N,\alpha})||_{\Dpa}$ is incorporated in the constant $C_p$
which appears in the final lines of the argument of the proof. In
particular,
$$
||\Delta_ng''||_{H^p}\leq C_p' \left(||\mathcal
H_g(f_{2^n,\alpha})||_{\Dpa}\right)  2^{n(2-\frac{1}{p})}
$$
therefore,
$$
\lim_{n\to\infty}\frac{||\Delta_n g''||_{H^p}}{2^{{n(2-\frac{1}{p})}}}=0
$$
so by Remark~\ref{re:mlip}, $g\in
\lambda\left(p,\frac{1}{p}\right)$.
\par Conversely, let  $\ep>0$ and    $g\in \lambda \left(p,\frac{1}{p}\right)$.
Suppose $\{f_k\}$ is a sequence of analytic functions in $\D$ satisfying (\ref{eq:com1})
and (\ref{eq:com2}). Then  there exists
$n_0\in\N$ such that
$$\frac{||\Delta_n g''||_{H^p}}{2^{n(2-\frac{1}{p})}}<\ep\quad\text{for all $n\ge n_0$}.
$$
Then it follows from the proof of Theorem~\ref{th:dirichlet}
(sufficiency part) that for all $k$
\begin{equation*}
\begin{split}\label{eq:csd1}
\n{\hg(f_k)}^p_{\Dpa}&\lesssim |\hg(f_k)(0)|^p+ \sum_{n=0}^\infty
2^{-n(\alpha+1)}\left(\int_0^1t^{2^{n-2}+1}|f_k(t)|\,dt\right)^p\left\|\Delta_n
g''\right\|^p_{H^p}.
\end{split}
\end{equation*}
Using Lemma \ref{le:c1} we see that
$$
|\hg(f_k)(0)|=\int_0^1 |f_k(t)|\,dt\to 0 \quad \mbox{as $k\to \infty$}.
$$
On the other hand
\begin{equation*}
\begin{split}
\sum_{n=0}^\infty 2^{-n(\alpha+1)}&\left(\int_0^1t^{2^{n-2}+1}|f_k(t)|\,dt\right)^p \left\|\Delta_n g''\right\|^p_{H^p}\\
 & \le  C\sum_{n=0}^{n_0-1} 2^{n(2p-2-\alpha)}\left(\int_0^1t^{2^{n-2}+1}|f_k(t)|\,dt\right)^p\\
  &\quad\quad +
C\ep \sum_{n_0}^\infty
2^{n(2p-2-\alpha)}\left(\int_0^1t^{2^{n-2}+1}|f_k(t)|\,dt\right)^p.
\end{split}
\end{equation*}
The finite sum above tend to $0$ as $k\to \infty$ by appealing to Lemma \ref{le:c1}. The second sum is
\begin{equation*}
\begin{split}
\sum_{n_0}^\infty &2^{n(2p-2-\alpha)}\left(\int_0^1t^{2^{n-2}+1}|f_k(t)|\,dt\right)^p\\
&\quad \quad \leq C\sum_{j=1}^{\infty}(j+1)^{(2p-3-\a)}\left(\int_0^1t^{j+1}|f_k(t)|\,dt\right)^p\\
  & \quad \quad\le C ||\hti(f_k)||^p_{\Dpa}\\
  & \quad \quad\le C \sup_{k}||f_k||_{\Dpa}\\
  &\quad \quad\leq CK
\end{split}
\end{equation*}
by (\ref{eq:com1}). This gives
$$
\lim_{k\to\infty}||\mathcal H_g(f_k)||_{\Dpa}\le CK\ep,
$$
and since $\ep$ is arbitrary the proof is complete.
\end{Pf}
\par\medskip Theorem \ref{th:hardyc} and Theorem \ref{th:bergmanc} can be proved
 with the same technique. We omit the details.
\par
 Finally, we shall prove Theorem~\ref{th:sachatten}.\par\medskip
\begin{Pf}{\em{ Theorem~\ref{th:sachatten}.}}
We recall that an operator $T$ on a separable Hilbert space $H$ is a
Hilbert-Schmidt operator if for an orthonormal basis
$\{e_n:n=0,1,2,\cdots\}$ of $H$  the sum
$\sum_{n=0}^{\infty}\n{T(e_n)}^2$ is finite. The finiteness of this
sum does not depend on the basis chosen. The class of
Hilbert-Schmidt operators on $H$ is denoted by $S^2(H)$.
\par
(i) The set  $\{1, z, z^2, \cdots, \}$ is a basis of $H^2$. If
$g(z)=\sum_0^{\infty} b_kz^k\in \hol(\D)$ then
$$
\hg(z^n)= \int_0^1t^ng'(tz)\,dt=
\sum_{k=0}^{\infty}\frac{(k+1)b_{k+1}}{n+k+1}z^k,
$$
thus
$$
\n{\hg(z^n)}_{H^2}^2 =\sum_{k=0}^{\infty}
\frac{(k+1)^2|b_{k+1}|^2}{(n+k+1)^2}
$$
and
\begin{align*}
\sum_{n=0}^{\infty}\n{\hg(z^n)}_{H^2}^2& =\sum_{n=0}^{\infty}\sum_{k=0}^{\infty}
\frac{(k+1)^2|b_{k+1}|^2}{(n+k+1)^2}\\
&=\sum_{k=0}^{\infty}(k+1)^2|b_{k+1}|^2
\sum_{n=0}^{\infty}
\frac{1}{(n+k+1)^2}\\
&\sim \sum_{k=0}^{\infty}(k+1)^2|b_{k+1}|^2\frac{1}{k+1}\\
&=\sum_{k=0}^{\infty}(k+1)|b_{k+1}|^2
\sim ||g||_{\mathcal{D}}^2.
\end{align*}
Thus $\mathcal{H}_g\in S^2(H^2)$ if and only if $g\in \mathcal{D}$.
\par (ii) On $A^2_\a$, $-1<\a<0$,  an orthonormal basis is
$$
\{e_n(z)=c_nz^n: n=0,1,2,\cdots,\}
$$
where
$$
 c_n=\frac{1}{\n{z^n}_{A^2_\a}}=\sqrt{\frac{\Gamma(n+2+\a)}{n!\Gamma(2+\a)}}
$$
Now
$$
\hg(e_n)(z)=c_n\hg(z^n)=c_n\sum_{k=0}^{\infty}\frac{(k+1)b_{k+1}}{n+k+1}z^k
$$
and
$$
\n{\hg(e_n)}_{A^2_\a}^2=c_n^2\sum_{k=0}^{\infty}\frac{k!\Gamma(2+\a)}{\Gamma(k+2+\a)}\frac{(k+1)^2|b_{k+1}|^2}{(n+k+1)^2}.
$$
Thus using the Stirling formula estimate $\frac{\Gamma(n+\b)}{n!}\sim (n+1)^{\b-1}$ we have
\begin{align*}
\sum_{n=0}^{\infty}\n{\hg(e_n)}_{A^2_\a}^2&
=\sum_{n=0}^{\infty}\sum_{k=0}^{\infty}
\frac{\Gamma(n+2+\a)}{n!\Gamma(2+\a)}\frac{k!\Gamma(2+\a)}{\Gamma(k+2+\a)}\frac{(k+1)^2|b_{k+1}|^2}{(n+k+1)^2}\\
&\sim \sum_{n=0}^{\infty}\sum_{k=0}^{\infty}\frac{(n+1)^{\a+1}}{(k+1)^{\a+1}}\frac{(k+1)^2|b_{k+1}|^2}{(n+k+1)^2}\\
&=\sum_{k=0}^{\infty}(k+1)^{1-\a}|b_{k+1}|^2\sum_{n=0}^{\infty}\frac{(n+1)^{\a+1}}{(n+k+1)^2}.
\end{align*}
Now a calculation shows that the asymptotic order of the inside series is
$$
\sum_{n=0}^{\infty}\frac{(n+1)^{\a+1}}{(n+k+1)^2}\sim (k+1)^\a ,
$$
and it follows that
$$
\sum_{n=0}^{\infty}\n{\hg(e_n)}_{A^2_\a}^2\sim
\sum_{k=0}^{\infty}(k+1)|b_{k+1}|^2 \sim ||g||_{\mathcal{D}}^2.
$$
\par
(iii) On $\mathcal D^2_\a$, $0<\a\leq 1$,  an orthonormal basis is
$$\{e_n\} =
\{1, d_1z, d_2z^2, \cdots,\}
$$
where
$$
 d_n=\frac{1}{\n{z^n}_{D^2_\a}}=\frac{1}{n}\sqrt{\frac{\Gamma(n-1+2+\a)}{(n-1)!\Gamma(2+\a)}}.
$$
In this case we find (omitting the details)
\begin{align*}
\sum_{n=0}^{\infty}\n{\hg(e_n)}_{D^2_\a}^2&\sim
\sum_{k=0}^{\infty}(k+1)^{(3-\a)}|b_{k+2}|^2\sum_{n=0}^{\infty}
\frac{(n+1)^{\a-1}}{(n+k+1)^2}\\
&\sim\sum_{k=0}^{\infty}(k+1)^{(3-\a)}|b_{k+2}|^2 (k+1)^{(\a-2)}\\
&\sim \sum_{k=0}^{\infty}(k+1)|b_{k+1}|^2
\sim ||g||_{\mathcal{D}}^2
\end{align*}
and the assertion follows.
\end{Pf}
\par\medskip

\end{document}